\documentclass[12pt]{amsart}
\usepackage{a4wide}
\usepackage{times}
\usepackage{bbm}
\usepackage{mathtools, amssymb}
\usepackage{graphicx,xspace}
\usepackage{epsfig}
\usepackage{dsfont}
\usepackage[usenames,dvipsnames]{xcolor}
\usepackage{tikz}
\usepackage[T1]{fontenc}
\usepackage[utf8]{inputenc}
\usepackage{array,multirow} %
\usepackage{bm}
\usepackage{kpfonts} 
\usepackage{dsfont} 
\usepackage{setspace} 
\usepackage{cancel}
\usepackage{tipa}
\usepackage{stmaryrd}
\usepackage{csquotes}

\usepackage{hyperref,pifont}
\usepackage{todonotes}
\usepackage{dsfont}
\newcommand{\comment}[1]{}

\newcommand{\ignorer}[1]{}

\def\and{\ \wedge\ }
\usepackage[capitalize]{cleveref}

\theoremstyle{plain}

\newtheorem{lemma}{Lemma}
\newtheorem{theorem}{Theorem}

\newtheorem{proposition}[lemma]{Proposition}

\newtheorem{definition}[lemma]{Definition}

\theoremstyle{remark}

\def\eps{\varepsilon}
\renewcommand\epsilon{\varepsilon}

\def\Kb{{K_3}}
\def\Kbb{{K_4}}
\def\Kab{{K'_2}}
\def\Kc{{K_5}}
\def\Kd{{K_6}}
\def\Kv{{K_7}}

\def\si{\sigma}
\def\ka{\kappa}
\def\la{\lambda}

\def\cI{\mathcal{I}}

\def\E{\mathbb{E}}

\def\bC{\mathbf{C}}
\def\bD{\mathbf{D}}

\def\One{\bm 1}
\newcommand{\MWST}[1]{\mathcal M(#1)}

\DeclareMathOperator{\Var}{Var}
\DeclareMathOperator{\Cov}{Cov}

\DeclareMathOperator{\Perm}{Perm}
\DeclareMathOperator{\loops}{loops}

\newcommand{\Old}[1]{}

\newcommand\restr[2]{{%
		\left.\kern-\nulldelimiterspace %
		#1 %
		\right|_{#2} %
	}}

\DeclareMathOperator{\cc}{cc}
\DeclareMathOperator{\CC}{CC}

\newcommand\esper{\mathbb E}

\title[Pattern counts in conjugacy classes]
{Asymptotic normality of pattern counts\\
in conjugacy classes}

 \author[V. Féray]{Valentin Féray}
       \address[VF]{Université de Lorraine, CNRS, IECL, F-54000 Nancy, France}
       \email{valentin.feray@univ-lorraine.fr}

 \author[M.S. Kammoun]{Mohamed Slim Kammoun}
    \address[MSK]{Université de Lyon, CNRS, ENSL, UMPA,  Lyon, France}
    \email{slim.kammoun@ens-lyon.fr}

\makeatletter
\@namedef{subjclassname@2020}{%
  \textup{2020} Mathematics Subject Classification}
\makeatother

\begin{document}
\maketitle
\begin{abstract}
    We prove, under mild conditions on fixed points and $2$-cycles,  
    the asymptotic normality of vincular pattern counts 
    for a permutation chosen uniformly at random in a conjugacy class. 
    Additionally, we prove that the limiting variance is always non-degenerate
for classical pattern counts.
    The proof uses weighted dependency graphs. 
\end{abstract}

\section{Introduction}

\subsection{Background}
A natural question in the context of random combinatorial structures
is the asymptotic behaviour of the number of substructures of a given type. In general, we expect such numbers to be asymptotically
normal or Poisson distributed, but proving it might be difficult.
For permutations, a good notion of substructures is that of {\em permutation patterns}.
Unsurprisingly, the asymptotic normality of patterns counts (either classical, consecutive or vincular patters)
in {\em uniform random permutations} has been extensively studied. 
We refer for example to \cite{BonaMonotonePatterns,JNZ15,Lisa,even-zohar2020patterns,janson2023normality} 
for such results. 

More recently, there has been some interest
in extending these results to other models of random permutations,
and in particular to uniform random permutations 
in some conjugacy classes.
In this spirit, asymptotic normality results have been obtained
for specific patterns:
see \cite{fulman1998descents-conjugacy,KimLee2020} for descent counts and \cite{fulman2022peaks} for the number of peaks 
(the latter is a sum of vincular pattern counts).
These results use minimal assumptions on the conjugacy class
under consideration.
In another direction, the second author \cite{kammoun2022universality} 
and independently Hamaker and Rhoades, \cite[Theorem 8.8]{hamaker} have both
proved asymptotic normality results for general patterns, 
but with some restrictions on the chosen conjugacy classes.
In the latter reference, the problem was raised of proving the asymptotic 
normality of pattern counts assuming only the convergence of the proportion of fixed points
and $2$-cycles in our permutations \cite[Problem 9.9]{hamaker}.
We solve this problem in the present paper; see \cref{thm:clt} below.
Additionally, we prove that the limiting Gaussian distribution is always non-degenerate
for classical pattern counts (\cref{thm:non-degeneracy}).
\medskip

{\em Note:} In parallel to this work, 
the asymptotic normality of {\em classical} pattern counts
 in random permutations in conjugacy classes
 has been obtained by Dubach \cite{dubach2024geometric}
by a different method.
His method only applies to classical patterns 
(and not to the general case of vincular patterns that we treat here), 
but can provide stronger results in this case (speed of convergence 
and large deviation estimates).
It can also be used to analyze other types of statistics on random 
permutations in conjugacy classes, such as longest monotone subsequences, the Robinson--Schensted shape,
and the number of records; all of these are out of reach with
the method of weighted dependency graphs used in this paper.

\subsection{Vincular patterns}
We use the notation $[n]$ to represent the set $\{1, 2, \ldots, n\}$.
 Additionally, we utilize the one-line notation for permutations
 $\sigma = \sigma_1 \sigma_2 \dots \sigma_n$.
\begin{definition}
  Let $\pi$ be a permutation of size $k$ and $A$ a subset of $[k-1]$.
  Then an {\em occurrence} of the vincular pattern $(\pi,A)$ in a permutation $\si$ is
  a tuple $(i_1,\cdots,i_k)$ with $i_1 < \cdots < i_k$ such that:
  \begin{itemize}
    \item $(i_1,\cdots,i_k)$ is an occurrence of $\pi$ in $\si$, i.e.    
      \[ \si_{i_{\pi^{-1}_1}} < \dots < \si_{i_{\pi^{-1}_k}}.\]
    \item for all $s$ in $A$, we have $i_{s+1}=i_s+1$.
  \end{itemize}
\end{definition}
In words, the subsequence $(\sigma_{i_1},\cdots,\sigma_{i_k})$
is in the same relative order as $(\pi_1,\dots,\pi_k)$ and $A$
indicates entries of the subsequence which must be consecutive
in $\sigma$.
For example, let us consider the permutation $\sigma=2\textcolor{blue}{17}345\textcolor{blue}{6}$ and the triple $(i_1, i_2, i_3)=(2,3,7)$
(the corresponding subsequence $(\sigma_{i_1}, \sigma_{i_2}, \sigma_{i_3})=(1,7,6)$ is in blue in $\sigma$).
\begin{itemize}
    \item The triple $(i_1, i_2, i_3)=(2,3,7)$ is an occurrence of $(\pi=132, \emptyset)$ since $\pi=132$ is the permutation with the same relative order as the subsequence $(\sigma_{i_1}, \sigma_{i_2}, \sigma_{i_3})=(1,7,6)$. Equivalently,
    $$\sigma_{i_{\pi^{-1}_1}} = \sigma_{i_1}= 1 < \sigma_{i_{\pi^{-1}_2}} = \sigma_{i_3} = 6 < \sigma_{i_{\pi^{-1}_3}} = \sigma_{i_2} = 7.$$

    \item The triple $(2,3,7)$ is also an occurrence of $(132, \{1\})$.
    Indeed, it is an occurrence of $132$ as explained before,
    and it satisfies $i_2=i_1+1$ (i.e.~the two first elements
    in the corresponding subsequence $(1,7,6)$ are adjacent in $\sigma$).

    \item The triple $(2,3,7)$ is however not an occurrence of $(132, \{2\})$ since $i_3 \neq i_2+1$; said otherwise, the elements $7$ and $6$ in
    the corresponding subsequence $(1,7,6)$ are not adjacent in $\sigma$.
\end{itemize}

Vincular patterns have been introduced in \cite{SteingrimssonGeneralizedPatterns}. 
They encompass the notions of classical and consecutive patterns. Indeed, when $A=\emptyset$, there are no constraints on indices being consecutive, and we recover the notion of classical patterns.
On the other hand, when $A=[k-1]$, all indices must be consecutive,
and we recover consecutive patterns.
Here are some examples of occurrences of vincular patterns in the literature.
\begin{itemize}
    \item An occurrence of $(21, \emptyset)$ is a classical occurrence of 21, {\em i.e.}~an inversion.
    \item An occurrence of $(21, \{1\})$ is a consecutive occurrence of 21, {\em i.e.}~a descent.
    \item An occurrence of $(123\dots k, \emptyset)$ is an increasing subsequence of length $k$.
    \item A peak is an occurrence of  $(132, \{1,2\})$ or $(231, \{1,2\})$. 
    Symmetrically, a valley is an occurrence of $(213, \{1,2\})$ or $(312, \{1,2\})$.
    \item The longest alternating subsequence of a permutation equals, up to an additive term equal to 1 or 2, the number of valleys, plus the number of peaks of the permutation. In other terms,
    it is almost a linear combination of vincular patterns (in that case, consecutive patterns).
     This result and some equivalent variants were obtained separately by different authors. We refer for example to Houdré and Restrepo~\cite{houdre2010} and Romik~\cite{romik2011alternating}.
     \item The generating function of the vincular pattern $(312,\{1\})$ 
     (together with other statistics) occurs as partition function of
     a model of statistical physics called PASEP (partially asymmetric simple exclusion process); see \cite{PASEP_31-2}.
\end{itemize}

\subsection{Our main result: asymptotic normality of vincular pattern counts
for uniform random permutation in conjugacy classes}
 Throughout this article, we fix a vincular pattern $(\pi,A)$ such that $\pi$ is of size $k$.
We let  $X^{(\pi,A)}(\sigma)$ be the number of the occurrences of $(\pi,A)$ in $\sigma$ and  we define $m$ by
$m=k-|A|$, which is the number of {\em blocks} in $(\pi,A)$. 
For an (integer) partition $\la$ of $n$, we let $\mathcal C_{\la}$ be the conjugacy class 
indexed by $\la$ in the symmetric group $S_n$, and  $\sigma_{\la}$ be a uniform random permutation in $\mathcal C_{\la}$. 
Our main result is the following statement,
where $m_i(\lambda^n)$ denotes the number of parts of size $i$ in $\lambda^n$.
\begin{theorem}\label{thm:clt}
   For any vincular pattern $(\pi,A)$,
   there exist two polynomial functions  $f$ and $g$  such that if $\frac{m_1(\la^{n})}{n} \to p_1 $  and $\frac{m_2(\la^{n})}{n} \to p_2 $, then  
    \begin{equation}\label{eq:clt}
    \frac{X^{(\pi,A)}(\sigma_{\la^n}) -
       \E(X^{(\pi,A)}(\sigma_{\la^n}))
        }{n^{m-\frac 12} } 
      \xrightarrow[n\to\infty]{d}
       \mathcal{N}(0,f(p_1)+p_2 g(p_1)).
       \end{equation}
       Moreover, we have convergence of all moments.
\end{theorem}
We prove  \cref{thm:clt} in \cref{sec:proof clt}. 
A difficulty, compared to the case of uniform random permutations in the whole symmetric
group $S_n$, is that, when working in a given conjugacy class,
occurrences of patterns at disjoint sets of positions are no longer independent events.
 Our proof thus relies then on a technique known as weighted dependency graphs introduced in \cite{WDG} in order to control higher cumulants for sums of random dependent variables. 
The main ingredient is \cref{prop:WDG_conjugacy} (proved in \cref{sec:WDG}),
 which states the existence of a weighted dependency graph structure 
 for  the random permutation $\sigma_{\lambda^n}$. 
 Both the proofs of \cref{prop:WDG_conjugacy} and of  \cref{thm:clt},
 starting from \cref{prop:WDG_conjugacy}, require delicate combinatorial
 arguments, to bound complicated multi-indexed sums.
 
We believe that \cref{prop:WDG_conjugacy} is interesting in itself. 
It can be used to prove that other classes 
of permutation statistics are asymptotically normal: for example, it is immediate to adapt the proof of \cref{thm:clt}
to the number of excedances\footnote{An excedance is a position $i$ such that  $\sigma(i)>i$.}, peaks and/or valleys
(and thus to the length of the longest alternating subsequence).
Also, the proof is readily adapted to establish joint convergence in \cref{thm:clt},
i.e.~ the limiting joint distribution of
$$\left(\frac{X^{(\pi^1,A^1)}(\sigma_{\la^n}) -
       \E(X^{(\pi^1,A^1)}(\sigma_{\la^n}))
        }{n^{m_1-\frac 12} },
   \dots, \frac{X^{(\pi^r,A^r)}(\sigma_{\la^n}) -
       \E(X^{(\pi^r,A^r)}(\sigma_{\la^n}))
        }{n^{m_r-\frac 12} } \right)
$$ is a multivariate Gaussian distribution for any family of patterns. 

\subsection{Non-degeneracy of the limiting Gaussian law for classical patterns}
A natural problem following \cref{thm:clt} is to determine whether
the limiting law is degenerate or not, i.e.~whether the limiting variance $f(p_1)+p_2 g(p_1)$
is positive or equal to $0$.
In the case of a uniform permutation in $S_n$, which corresponds\footnote{Uniform random permutations in the whole symmetric group $S_n$ are of course not a special case of random permutations in conjugacy classes. However, it follows easily from the results of \cite[Section 8]{hamaker} that the limiting variance for uniform random permutations is the same as for any conjugacy class with proportions of fixed points and $2$ cycles tending to $0$.} 
to $p_1=p_2=0$, it is known that the limiting variance, $f(0)$ in this case, 
is always positive; see \cite{Lisa,janson2023normality}.
Since $f$ and $g$ are polynomial functions, 
this implies that $f(p_1)+p_2 g(p_1)$ can only vanish 
when $(p_1,p_2)$ lives in a subvariety of codimension 1 of $\{(p_1,p_2) \in \mathbb R^2_+:\,  p_1+2p_2 \le 1\}$.
In other terms, we know that the limiting variance is generically positive.
It would nevertheless be desirable to have a complete characterization
of the degeneracy cases.

We could solve the question in the case of classical patterns, 
i.e.~the case $A=\emptyset$, 
for which we prove that the limiting Gaussian law is always non-degenerate
(except in the trivial case $p_1=1$, 
where most points of the permutation are fixed).
For simplicity, we write $\pi$ instead of $(\pi,\emptyset)$. Our second main result is the following.
\begin{theorem}
  For any {\bf classical} pattern $\pi$ and any
  $(p_1,p_2) \in \mathbb R_+^2$ with $p_1+2p_2 \le 1$, except $(p_1,p_2)=(1,0)$,
  we have $f(p_1) +p_2 g(p_1) >0$.
  \label{thm:non-degeneracy}
\end{theorem}
Note that in the case $p_1=1$, $p_2=0$, we can choose $\la^n=(1^n)$, 
i.e.~$\sigma_{\lambda^n}$ is a.s.~the identity permutation.
Thus, in this case, the limiting variance $f(1)$ is indeed $0$. 

The proof of this non-degeneracy statement is rather involved.
It uses a method introduced in
\cite{Lisa} and \cite{WDG_SetPartitions}.
The basic idea is to use a recursive construction of the random objects,
say $\sigma_n$, under consideration.
 Typically, such a recursive construction uses 
 a random object of smaller size, say $\sigma_{n-1}$,
 and an extra source of randomness, say $I$.
Then one can compute the variance of the quantity of interest, $X(\sigma_n)$,
 by conditioning on $I$ and using the law of total variance.
If the recursive construction is well-chosen,
then some terms in the law of total variance involve the variance of 
$X(\sigma_{n-1})$, {\it i.e.}, one can compute, or at least bound from below,
the variance recursively.
This is the basic principle, the details being quite subtle;
 see \cref{sec:non-degeneracy}.

\section{Preliminaries: weighted dependency graphs}
\label{sec:preliminaries-WDG}
In this section, we present the notion of weighted dependency graphs,
which is the main tool in the proof of \cref{thm:clt}.
The results presented in this section are taken from~\cite{WDG}.

\subsection{Definition}
\label{ssec:def_WDG}
We start by recalling the notion of mixed cumulants.
The mixed cumulant of a family of random variables $X_1,X_2,\ldots,X_r$ 
defined on the same probability space and having finite moments
is defined as
\begin{equation}
  \kappa(X_1,X_2,\ldots,X_r):=[t_1 \cdots t_r] \log \mathbb E[\exp(t_1 X_1 + \dots + t_r X_r) ].
  \label{eq:def_cumulants}
\end{equation}
If $X_i=X$ for all $i$, we abbreviate $\kappa(X_1,X_2,\ldots,X_r)$ as $\kappa_r(X)$.
This is the standard cumulant of a random variable.

An important property of cumulants is the following: if $\{X_1,X_2,\ldots,X_r\}$ can be written
as a disjoint union of two mutually independent sets 
of random variables, then $\kappa(X_1,X_2,\ldots,X_r)=0$.
Hence, cumulants can be seen as a kind of measure of dependency,
and the definition of weighted dependency graphs is based on these heuristics. Another important and elementary property of cumulants is that  $X$ follows the  normal distribution if and only if $\kappa_r(X)=0$ for every $r\geq 3$. 
\medskip

In the sequel, a weighted graph is a graph with weights on its {\em edges}, belonging to $(0,1]$. 
Non-edges can be interpreted as edges of weight $0$, so that a weighted graph can be equivalently seen
as an assignment of weights in $[0,1]$ to the edges of the complete graph.
All our definitions are compatible with this convention.

For a weighted graph $H$, we define  $\MWST{H}$
to be the {\em maximal weight of a spanning tree} of $H$,
the weight of a spanning tree being the product of the weights of its edges
(if $H$ is disconnected, there is no spanning tree
and as a consequence of the above convention, we have $\MWST{H}=0$).
For example, for $\eps<1$, the weighted graph $H$ of \cref{fig:exWDep}
satisfies $\MWST{H}=\eps^4$.
The following definition was introduced in \cite{WDG}.
\begin{figure}[ht]
    \begin{center}
        \includegraphics{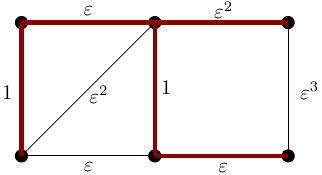}
    \end{center}
    \caption{Example of a weighted graph $H$
    with a marked spanning tree of maximum weight in red.}
    \label{fig:exWDep}
\end{figure}
\begin{definition}
\label{Def:WDG} Let $V$ be a finite set and
let $\{Y_v,v \in V\}$ be a family of random variables with finite moments, defined on the same probability space.
We fix a sequence $\bC=(C_1,C_2,\cdots)$ of positive real numbers
and a real-valued function  $\Psi$ defined on multisets of elements of $V$.

A weighted graph $L$ with vertex set $V$ is a $(\Psi,\bC)$ weighted dependency graph
for  $\{Y_v,v \in V\}$ if, 
for any multiset \hbox{$U=\{v_1,\ldots,v_r\}$} of elements of $V$,
one has
\begin{equation}
    \bigg| \ka\big( Y_v ; v \in U \big) \bigg| \le
    C_r \, \Psi(U) \, \MWST{L[U]}. 
    \label{EqFundamental}
\end{equation}
\end{definition}
Here, $L[U]$ is the weighted graph with vertex set $U$ and the same weights as $L$.
In examples of weighted dependency graphs, 
$\Psi$ and $\bC$ are simple or universal quantities,
so that the meaningful term is $\MWST{L[U]}$.
Note that the smaller the weights on edges are,
the smaller $\MWST{L[U]}$ is, 
{\it i.e.}~the closer to independence are the corresponding random variables $\{Y_v,v \in U\}$.
Hence, the edge weights in a weighted dependency graph
quantify in some sense the dependency between the random variables $\{Y_v,v \in V\}$.

\subsection{A useful example: uniform random permutations in $S_n$}
\label{sec:WDG_uniform_permutation}
Let $n \ge 1$ be an integer, and $\tau$ be a uniform random permutation in the symmetric group $S_n$.
For $(i,j)$ in $[n]^2$, we let $A_{i,j}=\One[\tau_i=j]$.
Note that variables $A_{i,j}$ and $A_{k,\ell}$ cannot be both equal to 1 if $(i,j) \ne (k,\ell)$
but either $i=k$ or $j=\ell$. 
If $i\ne k$ and $j \ne \ell$, we have
\[\mathbb P[A_{i,j} A_{k,\ell}]=\frac{1}{n(n-1)},\]
which is close but different from $\mathbb P[A_{i,j}] \cdot \mathbb P[A_{k,\ell}]=1/n^2$.
Hence, variables $A_{i,j}$ and $A_{k,\ell}$ are weakly dependent.
It turns out that this family
of random variables admits a nice weighted dependency graph.
This was shown in \cite[Proposition 8.1]{WDG},
which we copy here.
(We use the symbol $\#$ for the number of distinct elements in a set or multiset.)
\begin{proposition}\label{prop:WDG_uniform_permutations}
    Consider the weighted graph $L$ on vertex set $[n]^2$ defined as follows:
    \begin{itemize}
        \item if two pairs $v_1=(i_1,j_1)$ and $v_2=(i_2,j_2)$
            satisfy either $i_1=i_2$ or $j_1=j_2$,
            then they are linked in $L$ by an edge of weight $1$.
        \item otherwise,
            they are linked in $L$ by an edge of weight $1/n$.
    \end{itemize}

    Then $L$
    is a $(\Psi,\bC)$ weighted dependency graph,
    for the family $\{ Y_{i,j},\, (i,j) \in [n]^2 \}$,
    where
    \begin{itemize}
        \item $\Psi(U)=n^{-\#(U)}$ for any multiset $U$ of elements of $[n]^2$;
        \item $\bC=(C_r)_{r \ge 1}$ is a sequence that does not depend on $n$.
    \end{itemize}  
  \end{proposition}

\subsection{Cumulant bounds}
Weighted dependency graphs are useful to bound cumulants and prove asymptotic normality results.
In particular, in this paper, we will use the following lemma \cite[Lemma 4.10]{WDG}
\begin{lemma}
    Let $L$ be a $(\Psi,\bC=(C_r)_{r\geq 1})$ weighted dependency graph
    for a family of variables $\{Y_v, v \in V\}$.
    Define $R$ and $T_\ell$ (for $\ell \ge 1$) as follows:
\begin{align}
    R &= \sum_{v \in V} \Psi(\{v\});
    \label{EqDefR}\\
    \label{EqDefT}
    T_\ell &= \max_{\substack{v_1,\ldots,v_{\ell} \in V}} \left[ \sum_{u \in V} 
    \left( \max_{i \in [\ell]} w(u,v_i) \right) \cdot
    \frac{\Psi\big(\{v_1,\cdots,v_{\ell},u\}\big)}{\Psi\big(\{v_1,\cdots,v_{\ell} \}\big)}   \right],
\end{align}
where $w(u,v)$ is the weight of the edge $(u,v)$ in $L$. 
    Then, for $r \ge 1$, we have
    \[ \left| \ka_r \left( \sum_{v \in V} Y_v \right) \right|
    \le   C_r\, r! \, R\, T_1 \cdots T_{r-1}.   \] 
    \label{lem:BorneCumulant}
\end{lemma}

\subsection{Power of weighted dependency graphs}
\label{ssec:product_WDG}
An important property of weighted dependency graphs is the following stability property.
Consider a family of random variables $\{Y_v,v \in V\}$ 
with a $(\Psi,\bC)$ weighted dependency graph $L$
and fix some integer $d \ge 1$.
We are interested in monomials $\bm Y_I:=\prod_{v \in I} Y_{v}$ of degree at most $d$,
i.e.~where $I \in V^d$ is a list of elements of $V$ of length $d$ (possibly with repetitions).
This defines a new family of random variables $\{\bm Y_I, I \in V^d\}$.
It turns out that this family admits a natural weighted dependency graph inherited from that of $\{Y_v,v \in V\}$.

To state this formally, we need to introduce some more notation.
Let $L$ be a weighted graph with vertex set $V$.
The $d$-th power $L^d$ of $L$ is a weighted graph with vertex set $V^d$.
If $I$ and $J$ are in $V^d$, the weight between vertices $I$ and $J$ 
is \[ W(I,J):= \max_{u \in I,v \in J} w(u,v),\]
where $w(u,v)$ is the weight between $u$ and $v$ in $L$.

Besides, we say that a function $\Psi$ defined on multisets of elements of a set $V$
is super-multiplicative if $\Psi(U_1 \uplus U_2) \ge \Psi(U_1) \Psi(U_2)$
for all multisets $U_1$ and $U_2$ (here and throughout the paper, $\uplus$ is the {\em disjoint union} of multisets,
i.e.~if an element belongs to both $U_1$ and $U_2$ with multiplicity $m_1$ and $m_2$, it belongs to $_1 \uplus U_2$
with multiplicity $m_1+m_2$).
Finally, a function $\Psi$ defined on lists of $V$
induces a function on lists of multisets of elements of $V$ by setting
\begin{equation}
     \bm{\Psi}( (I_1,\cdots,I_r)) = \Psi(I_1 \uplus \cdots \uplus I_r).
\end{equation}

The following proposition was proved in \cite[Proposition 5.11]{WDG}.
\begin{proposition}
  \label{prop:WDG_Products}
  Consider random variables $\{Y_v,v \in V\}$
  with a $(\Psi,\bC)$ weighted dependency graph $L$
  and assume that $\Psi$ is super-multiplicative. 
  
  We fix some integer $d \ge 1$. 
  Then $L^d$ is a $(\bm{\Psi},\bD_d)$ weighted dependency graph for the family
  $\{ \bm Y_I,\, I \in V^d\}$,
  where the constants $\bD_d=(D_{d,r})_{r \ge 1}$ depend only on
  $d$, $r$ and $\bC$.
\end{proposition}

\section{A weighted dependency graph structure\\
for uniform random permutation in conjugacy classes} 
\label{sec:WDG}

Fix an integer partition $\lambda^n$ of size $n$, for $(i,j)$ in $[n]^2$, we introduce the random variables
\[ B_{i,j} = \One[\sigma_{\la^n}(i)=j ].\] 

Note that the law of $B_{i,j}$  depends on $\la^n$.
Given a subset $\alpha=\{(i_t,j_t), t \le T\}$
 of $[n]^2$, we write $S(\alpha)=\{i_1,\ldots,i_T,j_1,\ldots,j_T\}$
 for the {\em support} of $\alpha$.
 Furthermore, we denote by $\cc(V,E)$  the number of connected components of the graph $(V,E)$.
 We also write $\CC(\alpha)$ instead of $\cc(S(\alpha),\alpha)$;
this is the number of components of the graph with edge set $\alpha$ and no isolated vertices. 
We now let
\begin{equation}\label{eq:def_Psi}
\Psi(\alpha) = n^{CC(\alpha)-\# S(\alpha)}.
\end{equation}
We define similarly $\Psi$ on lists of elements of $[n]^2$,
and it is insensitive to the order and repetitions of arguments.

 The choice of $\Psi$ may seem surprising, but in fact, it is quite natural. Indeed, the joint moments $M_\alpha:=\mathbb{E}(\prod_{(i,j) \in \alpha} B_{i,j})$ are nonzero if and only if $(S(\alpha),\alpha)$ is a disjoint union of directed paths and cycles. 
 In this case, $M_\alpha$ is $O(\Psi(\alpha))$, and this is optimal 
 (in the sense that $M_\alpha=\Theta(\Psi(\alpha))$ if $m_k(\la)=\Theta(n)$ for any fixed $k$).
 Furthermore, as proved in the next lemma,
  $\Psi$ has been chosen to be super-multiplicative,
 so that we can use \cref{prop:WDG_Products}.
\begin{lemma}
  The function $\Psi$ defined by \eqref{eq:def_Psi} is super-multiplicative.
\end{lemma}
\begin{proof}
  Let $\alpha=\{(i_t,j_t), t \le T\}$  and $\alpha'=\{(i'_t,j'_t), t \le T'\}$.
  We want to prove that
  \begin{multline}
\label{eq:supermultiplicativity}
    \CC(\alpha \cup  \alpha') + \#\{i_1,\ldots,i_T,j_1,\ldots,j_T \}+ \#\{i'_1,\ldots,i'_{T'},j'_1,\ldots,j'_{T'} \}  \\ \ge \CC(\alpha) + \CC(\alpha')+ \#\{i_1,\ldots,i_T,j_1,\ldots,j_T,i'_1,\ldots,i'_{T'},j'_1,\ldots,j'_{T'} \} .
  \end{multline}
  When $\alpha$ contains a cycle $C$ of length at least $2$, removing an edge of $C$ from $\alpha$
  does not change any quantity of \eqref{eq:supermultiplicativity}.
By symmetry, we may assume that neither $\alpha$ nor $\alpha'$ contain a cycle, except possibly loops.
Then, writing $\loops(\alpha)$ for the number of $\loops$ in $\alpha$ (counted with multiplicity), one has
$$ \CC(\alpha)+T = \#\{i_1,\ldots,i_T,j_1,\ldots,j_T \} + \loops (\alpha)$$
      and
      $$ \CC(\alpha')+T' = \#\{i'_1,\ldots,i'_T,j'_1,\ldots,j'_T \}+ \loops (\alpha').$$
Consequently,  \eqref{eq:supermultiplicativity} is equivalent to
\[    \CC(\alpha \cup  \alpha') +T+T'  \ge  \#\{i_1,\ldots,i_T,j_1,\ldots,j_T,i'_1,\ldots,i'_T,j'_1,\ldots,j'_T \}+ \loops (\alpha)+  \loops (\alpha') .
  \]
The latter holds true since for any multigraph $G=(V,E)$, one has
\[\cc(G) + |E| \ge |V| + \loops(E).\qedhere\]
\end{proof}

The goal of this section is to prove the following proposition.

\begin{theorem}
  \label{prop:WDG_conjugacy}
  There is a universal sequence $\bm C=(C_r)_{r \ge 1}$ such that the following holds.
  For any $n \ge 1$ and any $\la \vdash n$,
  the complete graph $L$ on $[n]^2$ with weights
  \begin{equation}\label{eq:def_w}
    w\big( (i,j),(k,\ell) \big)=\begin{cases}
    \frac1n &\text{ if $\{i,j\}$ and $\{k,\ell\}$ are disjoint;}\\
    1 &\text{otherwise,}
  \end{cases}
\end{equation}
  is a $(\Psi,\bm C)$ weighted dependency graph for the family $\{B_{i,j},(i,j) \in [n]^2\}$.
\end{theorem}
Fix $r \ge 1$ and let $\alpha=\{\{(i_t,j_t), 1 \le t \le r\}\}$,
be a multiset of size $r$ of elements of $[n]^2$.
We want to show that
\begin{equation}\label{eq:to_show_not_simplified}
  \left|\kappa \big( B_{i_1,j_1},\dots,B_{i_r,j_r}\big)\right| \le C_r \, \Psi(\alpha)\, \mathcal M(L[\alpha]),
\end{equation}
where $\mathcal M(L[\alpha])$ is defined in \cref{ssec:def_WDG}.
We start with a lemma.
\begin{lemma}
\label{lem:ML_cc}
For any $\alpha=\{(i_1,j_1),\dots,(i_r,j_r)\}$, we have $\mathcal M(L[\alpha])=n^{-\CC(\alpha)+1}$.
\end{lemma}
\begin{proof}
By definition, the weighted graph $L$, and therefore also $L[\alpha]$,
have only edges with weight $1$ or $n^{-1}$.
To form a spanning tree of maximal weight, we need to use as many edges of weight $1$
as possible.
It follows that $\mathcal M(L[\alpha])=n^{-c+1}$, where $c$ is the number of
connected components
of the subgraph $L_1[\alpha]$ of $L[\alpha]$ obtained by keeping only edges of weight $1$.

It remains to relate connected components of $L_1[\alpha]$ to those of $(S(\alpha),\alpha)$.
By construction, $L_1[\alpha]$ has $r$ vertices decorated with pairs $(i_t,j_t)$
and edges between pairs with a non-empty intersection.
Consider a variant, denoted $K_1[\alpha]$, where we have $2r$ vertices with decorations $i_t$ and $j_t$ ($t$ running from $1$ to $r$) and with two kinds of edges
\begin{itemize}
\item for each $t$, the vertex decorated with $i_t$ is connected to that decorated with $j_t$;
\item vertices with equal decorations are connected.
\end{itemize}
See \cref{fig:example_3graphs} for an example.
Contracting the first type of edges in $K_1[\alpha]$ gives $L_1[\alpha]$.
Contracting the second type of edges instead gives the graph $(S(\alpha),\alpha)$ considered above. Hence, both have the same number of connected components and the lemma follows.
\end{proof}

\begin{figure}
\[\includegraphics{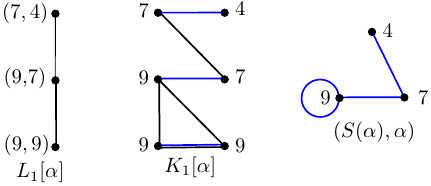}\]
\caption{From left to right: the three graphs $L_1[\alpha]$, $K_1[\alpha]$
and $(S(\alpha),\alpha)$ associated
with the $r=3$, $i_1=j_2=7$, $j_1=4$, $i_2=i_3=j_3=9$.
Edges of the first kind in $K_1[\alpha]$ are plotted in blue.
In this case, all three graphs are connected.}
\label{fig:example_3graphs}
\end{figure}
Using \cref{lem:ML_cc} and the definition of $\Psi$ (\cref{eq:def_Psi}),
\cref{eq:to_show_not_simplified} reduces to
\begin{equation}\label{eq:to_show}
 \left| \kappa \big( B_{i_1,j_1},\ldots,B_{i_r,j_r}\big) \right|\le C_r n^{-\#\{i_1,\ldots,i_r,j_1,\ldots,j_r\}+1}.
\end{equation}

We now prove this inequality.
The idea is to represent $\sigma_\la$ as $\tau^{-1} \circ \rho \circ \tau$, 
where $\rho$ is a fixed permutation in $\mathcal C_\la$ and $\tau$ a uniform random
permutation in $S_n$. As in \cref{sec:WDG_uniform_permutation},
we write $A_{i,j}=\One[\tau_i=j]$ and we have
\[B_{i,j} = \sum_{k=1}^n A_{i,k} A_{j,\rho_k}.\]
By multilinearity of cumulants, one has
\begin{equation}\label{eq:expanding_kappa}
  \kappa \big( B_{i_1,j_1},\ldots,B_{i_r,j_r}\big)
= \sum_{1 \le k_1,\ldots,k_r \le n}
\kappa \big( A_{i_1,k_1} A_{j_1,\rho_{k_1}},\ldots,A_{i_r,k_r} A_{j_r,\rho_{k_r}} \big).
\end{equation}
Combining \cref{prop:WDG_uniform_permutations} together with \cref{prop:WDG_Products} for $m=2$,
we know that $L^2$ is a $(\Psi,\bD)$ weighted dependency graph
for the family $\{A_{i,k} A_{j,k'}, (i,k,j,k') \in [n]^4\}$,
where
\begin{itemize}
\item $L^2$ is the second power of the graph $L$ defined in 
\cref{prop:WDG_uniform_permutations}; namely, in the graph $L^2$,
the vertex associated with $A_{i_s,k_s} A_{j_s,k'_s}$
is linked to that of $A_{i_t,k_t} A_{j_t,k'_t}$ by an edge of weight $1$
if \hbox{$\{i_s,j_s\} \cap \{i_t,j_t\} \ne \emptyset$} or $\{k_s,k'_s\} \cap \{k_t,k'_t\} \ne \emptyset$,
and by an edge of weight $1/n$ otherwise.
\item $\Psi\big( (i_1,k_1,j_1,k'_1), \dots, (i_r,k_r,j_r,k'_r) \big)
= n^{-\# \{(i_t,k_t),\, t\le r\} \cup \{(j_t,{k'_t}),\, t \le r \} }$;
\item $\bm D=(D_r)_{r \ge 1}$ is a universal sequence.
\end{itemize}
In particular, for any $(i_1,\dots,i_r)$, $(j_1,\dots,j_r)$ and $(k_1,\dots,k_r)$,
we have
\begin{multline}\label{eq:bound_kappaA}
  \Big| \kappa \big( A_{i_1,k_1} A_{j_1,\rho_{k_1}},\ldots,A_{i_r,k_r} A_{j_r,\rho_{k_r}} \big) \Big|
  \le D_r \cdot n^{-\# \{(i_t,k_t),\, t\le r\} \cup \{(j_t,\rho_{k_t}),\, t \le r \} } \\
\cdot \mathcal M\big(L^2\big[ (i_1,k_1,j_1,\rho_{k_1}), \dots, (i_r,k_r,j_r,\rho_{k_r})\big] \big).
\end{multline}
The weighted graph $L^2$ has only edges with weight $1$ and $1/n$.
Therefore for any subgraph $L^2[U]$, we have
$\mathcal M(L^2[U])=n^{-\cc(L^2_1[U])+1}$,
where $L^2_1[U]$ is the subgraph of $L^2[U]$ formed by its edges of weight 1.

In the case we are interested in,
we get
\begin{equation}\label{eq:ML2}
\mathcal M\big(L^2\big[ (i_1,k_1,j_1,\rho_{k_1}), \dots, (i_r,k_r,j_r,\rho_{k_r})\big] \big)
= n^{-\cc(G^A_{k_1,\ldots,k_r})+1},
\end{equation}
where $G^A_{k_1,\ldots,k_r}$ is the graph with vertex set $[r]$
and, for each $s,t$ in $[r]$, it has an edge between $s$ and $t$ if
and only if
  \hbox{$\{i_s,j_s\} \cap \{i_t,j_t\} \ne \emptyset$}
   or $\{k_s,\rho_{k_s}\} \cap \{k_t,\rho_{k_t}\} \ne \emptyset$.
We note that $G^A_{k_1,\ldots,k_r}$ depends on $i_1,\dots,i_r,j_1,\dots,j_r$,
but since these indices are fixed throughout the proof, we keep this dependence
implicit in the notation. Only the dependence in $k_1,\dots,k_r$
is made explicit.

Let us illustrate this definition with an example. \label{example}
We take $\lambda=(6,2,1)$ and choose $\rho$ to be $(1,2,3,4,5,6)(7,8)(9)$
(in the product of cycle notation).
As in \cref{fig:example_3graphs}, we let $r=3$, $i_1=j_2=7$, $j_1=4$, $i_2=i_3=j_3=9$.
We consider the term indexed by $(k_1,k_2,k_3)=(9,5,4)$, yielding $(\rho_{k_1},\rho_{k_2},\rho_{k_3})=(9,6,5)$.
Then $1$ is linked to $2$ in $G^A_{k_1,\ldots,k_r}$ because $\{i_1,j_1\} \cap \{i_2,j_2\}=\{7\}\ne \emptyset$. Likewise $\{i_2,j_2\} \cap \{i_3,j_3\}=\{9\}\ne \emptyset$
and $\{k_2,\rho_{k_2}\} \cap \{k_3,\rho_{k_3}\} = \{5\} \ne \emptyset$,
implying that $2$ is linked to $3$ in $G^A_{k_1,\ldots,k_r}$
(in fact, only one of these intersections being nonempty would suffice for $2$ to be linked to $3$).
One can check however that $1$ is not linked to $3$
since $\{i_1,j_1\} \cap \{i_3,j_3\}= \{4,7\} \cap \{9\} =\emptyset$
and $\{k_1,\rho_{k_1}\} \cap \{k_3,\rho_{k_3}\} = \{9\} \cap \{4,5\}=\emptyset$.
Hence, $G^A_{k_1,\ldots,k_r}$ is the path $\begin{array}{c} \includegraphics{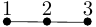}
\end{array}$.

To help us in our analysis, we now introduce, for each $r$-tuple $k_1,\cdots,k_r$, 
four (loop-free) graphs denoted $G^{(1)}_{k_1,\ldots,k_r}$, $G^{(2)}_{k_1,\ldots,k_r}$,
$G^\vee_{k_1,\ldots,k_r}$ and $G^\wedge_{k_1,\ldots,k_r}$,
all on vertex set $[r] \times \{1,2\}$.
For all these graphs, the idea is to label the vertex $(s,1)$ with the pair $(i_s,k_s)$
and $(s,2)$ with the pair $(j_s,\rho_{k_s})$.
Then we put an edge between two vertices labelled by $(x,y)$ and $(x',y')$
\begin{itemize}
  \item in $G^{(1)}_{k_1,\ldots,k_r}$ if $x=x'$, i.e.~if the first coordinates in the labels coincide;
  \item in $G^{(2)}_{k_1,\ldots,k_r}$ if $y=y'$, i.e.~if the second coordinates in the labels coincide;
  \item in $G^\wedge_{k_1,\ldots,k_r}$ if $(x,y)=(x',y')$,  i.e.~if the labels coincide;
    \item in $G^\vee_{k_1,\ldots,k_r}$ if either $x=x'$ or $y=y'$, i.e.~if one coordinate in the labels coincide.
\end{itemize}
Additionally, in $G^{(2)}_{k_1,\ldots,k_r}$ and $G^\vee_{k_1,\ldots,k_r}$,
we add edges between $(s,1)$ and $(s,2)$ for any $s \le r$ (for clarity,
these extra edges will be drawn in blue).
We continue our example above by representing the four associated graphs
$G^{(1)}_{k_1,\ldots,k_r}$, $G^{(2)}_{k_1,\ldots,k_r}$,
$G^\vee_{k_1,\ldots,k_r}$ and $G^\wedge_{k_1,\ldots,k_r}$ in \cref{fig:example_4graphs}.
We note that a coincidence between a value in $\{i_1,\dots,i_r,j_1,\dots,j_r\}$
and a value in the set $\{k_1,\dots,k_r,\rho_{k_1},\dots,\rho_{k_r}\}$ 
has no influence on any of the graphs considered here.
\begin{figure}
 \[\includegraphics{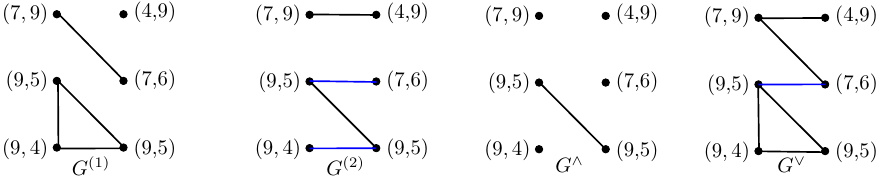}\]
\caption{From left to right: the four graphs $G^{(1)}_{k_1,\ldots,k_r}$, $G^{(2)}_{k_1,\ldots,k_r}$,
$G^\vee_{k_1,\ldots,k_r}$ and $G^\wedge_{k_1,\ldots,k_r}$ associated
with the particular values of $i_1,\dots,i_r,j_1,\dots,j_r,k_1,\dots,k_r$ 
given on page \pageref{example}.}
\label{fig:example_4graphs}
\end{figure}

Comparing their definitions,
we immediately see that $G^A_{k_1,\ldots,k_r}$ can be obtained from
$G^\vee_{k_1,\ldots,k_r}$ by merging, for each $s \le r$,
the vertices $(s,1)$ and $(s,2)$ into a single vertex $s$.
Since $G^\vee_{k_1,\ldots,k_r}$ has edges between  $(s,1)$ and $(s,2)$ for any $s \le r$,
this merge operation does not change the number of connected components. 
Besides, again directly from definition,
we see that connected components of  $G^\wedge_{k_1,\ldots,k_r}$
correspond to distinct pairs in the set $\{(i_t,k_t),\, t\le r\} \cup \{(j_t,\rho_{k_t}),\, t \le r\}$.

With these observations,
 \cref{eq:bound_kappaA,eq:ML2} imply
\begin{equation}\label{eq:first_bound_kappaA}
 \Big| \kappa \big( A_{i_1,k_1} A_{j_1,\rho_{k_1}},\dots,A_{i_r,k_r} A_{j_r,\rho_{k_r}} \big) \Big|
\le D_r\,  n^{-\cc(G^\wedge_{k_1,\ldots,k_r})-\cc(G^\vee_{k_1,\ldots,k_r})+1}.
\end{equation}

We now need the following easy, possibly well-known, graph theoretical lemma.
\begin{lemma}
  Let $V$ be a vertex set and $E_1$, $E_2$ be two sets of edges on $V$.
  Then
  \begin{equation}\label{eq:Inclusion_Exclusion_Connected_Component}
    \cc(V,E_1 \cup E_2) + \cc(V,E_1 \cap E_2) \ge \cc(V,E_1) + \cc(V,E_2).
  \end{equation}
\label{lem:Inclusion_Exclusion_Connected_Component}
\end{lemma}
\begin{proof}
  Assume $E_1$ contains a cycle $C$.
  If $E_2$ contains $C$ as well, we can remove any edge of $C$ from $E_1$
  without changing any of the quantities appearing in \eqref{eq:Inclusion_Exclusion_Connected_Component}.
  If $E_2$ does not contain $C$, removing from $E_1$ an edge of $C$ which is not in $E_2$
  does not change any of the quantities appearing in \eqref{eq:Inclusion_Exclusion_Connected_Component} either.
  Hence, it is enough to prove \eqref{eq:Inclusion_Exclusion_Connected_Component}
  when $E_1$ does not contain any cycle.
  With the same argument, we can assume that $E_2$ does not contain any cycle as well.
  In this case, we have
  \[\cc(V,E_1)= |V|- |E_1|, \quad \cc(V,E_2)= |V|- |E_2|, \quad \cc(V,E_1 \cap E_2) =|V|-|E_1 \cap E_2|.\]
  The union $E_1 \cup E_2$ may however contain cycles, so we only have an inequality
  \[\cc(V,E_1 \cup E_2) \ge |V|- |E_1 \cup E_2|. \]
  Together with the standard inclusion-exclusion principle $|E_1 \cup E_2| +|E_1 \cap E_2|= |E_1| + |E_2|$,
  this proves the lemma.
\end{proof}

By construction,
the graph $\cc(G^\vee_{k_1,\ldots,k_r})$ is the union of $G^{(1)}_{k_1,\ldots,k_r}$ and $G^{(2)}_{k_1,\ldots,k_r}$.
On the other hand, the edges of $G^\wedge_{k_1,\ldots,k_r}$ are included in that
of the intersection of $G^{(1)}_{k_1,\ldots,k_r}$ and $G^{(2)}_{k_1,\ldots,k_r}$
(since we have added edges $\{(s,1),(s,2)\}$ in $G^{(2)}_{k_1,\ldots,k_r}$, it might happen
that such an edge is in both $G^{(1)}_{k_1,\ldots,k_r}$ and $G^{(2)}_{k_1,\ldots,k_r}$,
but not in  $G^\wedge_{k_1,\ldots,k_r}$; this is for example the case
of the bottom edge in \cref{fig:example_4graphs}).
Consequently, $G^\wedge_{k_1,\ldots,k_r}$ has more connected components than this intersection.
Combining these observations with
 \cref{lem:Inclusion_Exclusion_Connected_Component}, for any $k_1,\ldots,k_r$, we have
\begin{align*}
\cc(G^\wedge_{k_1,\ldots,k_r})+\cc(G^\vee_{k_1,\ldots,k_r}) &\ge
  \cc(G^{(1)}_{k_1,\ldots,k_r} \cap G^{(2)}_{k_1,\ldots,k_r}) +  \cc(G^{(1)}_{k_1,\ldots,k_r} \cup G^{(2)}_{k_1,\ldots,k_r}) \\
  & \ge \cc(G^{(1)}_{k_1,\ldots,k_r}) +\cc(G^{(2)}_{k_1,\ldots,k_r}).
  \end{align*}

Back to \eqref{eq:first_bound_kappaA}, we get that, for any $k_1,\ldots,k_r$,
\begin{equation}
\Big| \kappa \big( A_{i_1,k_1} A_{j_1,\rho_{k_1}},\dots,A_{i_r,k_r} A_{j_r,\rho_{k_r}} \big) \Big|
\le D_r 
n^{-\cc(G^{(1)}_{k_1,\ldots,k_r}) -\cc(G^{(2)}_{k_1,\ldots,k_r})+1}.
  \label{eq:second_bound_kappaA}
\end{equation}
At this stage, we observe the graph $G^{(1)}_{k_1,\ldots,k_r}$ is independent of $k_1,\ldots,k_r$
(it compares the first coordinates of the labels, which are either $i$'s or $j$'s).
Its number of connected components is the number of distinct elements in $\{i_1,\ldots,i_r,j_1,\ldots,j_r\}$.
Hence, summing \eqref{eq:second_bound_kappaA} over  $k_1,\ldots,k_r$ and recalling \eqref{eq:expanding_kappa},
we obtain
\begin{equation}
\Big| \kappa \big( B_{i_1,j_1},\ldots,B_{i_r,j_r}\big) \Big|
\le D_r\cdot n^{-\#\{i_1,\ldots,i_r,j_1,\ldots,j_r\}+1} \cdot \sum_{1 \le k_1,\ldots,k_r \le n}
n^{-\cc(G^{(2)}_{k_1,\ldots,k_r})}.
  \label{eq:third_bound_kappaB}
\end{equation}

We need a final lemma.
\begin{lemma}
Fix $\rho$ in $S_n$. For any graph $G$ on vertex set $[r]\times\{1,2\}$,
the number of lists $(k_1,\cdots,k_r)$ in $[n]^r$ with $G^{(2)}_{k_1,\ldots,k_r}=G$ is at most $n^{\cc(G)}$.
  \label{lem:counting}
\end{lemma}
\begin{proof}
 Observe that if  $G^{(2)}_{k_1,\ldots,k_r}=G$ and if $(t,h)$ is linked to $(s,h')$ in $G$,
 the value of $k_s$ is determined by that of $k_t$; they should either be equal if $h=h'$, or they should satisfy $k_s=\rho_{k_t}$ if $h=1$ and $h'=2$, or $k_t=\rho_{k_s}$ if $h=2$ and $h'=1$.
 Hence, a list $(k_1,\cdots,k_r)$ in $[n]^r$ with $G^{(2)}_{k_1,\ldots,k_r}=G$
 is determined by the value of one $k_t$ in each component of $G$.
 This proves the lemma.
\end{proof}
Consequently, for any graph $G$ on vertex set  $[r]\times\{1,2\}$, the lists $(k_1,\cdots,k_r)$ in $[n]^r$ with $G^{(2)}_{k_1,\ldots,k_r}=G$
contributes at most $1$ to the sum $\sum_{1 \le k_1,\ldots,k_r \le n}
n^{-\cc(G^{(2)}_{k_1,\ldots,k_r})}$.
Thus the latter sum is bounded by a constant depending only on $r$, namely $2^{\binom{2r}{2}}$.
Inserting this inequality in \eqref{eq:third_bound_kappaB},
this proves \eqref{eq:to_show} and thus \cref{prop:WDG_conjugacy}.

\section{Proof of asymptotic normality of vincular pattern counts}
\label{sec:proof clt}
By the method of moments, convergence \eqref{eq:clt}, together with moment convergence, hold if 
\begin{itemize}
\item $Var( X^{\pi}(\sigma_{\la^n}) ) = n^{2m-1} ( f(p_1)+p_2 g(p_1)) (1+o(1))$;
\item and higher cumulants tend to $0$ after normalization, i.e.~$\kappa_r(X^{\pi}(\sigma_{\lambda^n})) =o(n^{r(m-\frac 12)})$ for $r\ge 3$.
\end{itemize}
The existence of polynomial functions $f$ and $g$ such that the first item holds
is proven in \cite{hamaker} (see Proposition 7.2 and Theorem 8.14 there).
We will bound the cumulants using the weighted dependency graph structure of \cref{sec:WDG}.

Note that we have
\[X^{(\pi,A)}(\sigma_{\la^n})= \sum_{i_1 < \dots < i_k \atop i_{s+1}=i_s+1 \text{ for }s \in A} \,
\sum_{j_1,\dots,j_k \atop j_{\pi^{-1}(1)} < \dots < j_{\pi^{-1}(k)} } B_{i_1,j_1} \cdots B_{i_k,j_k}. \]
We denote by $\cI^{(\pi,A)}$ the set of tuples $(i_1,\dots,i_k,j_1,\dots,j_k)$ as in the above sum.
To bound the cumulant of $X^{(\pi,A)}(\sigma_{\la^n})$, we
need a weighted dependency graph for the family of products $B_{i_1,j_1} \cdots B_{i_k,j_k}$. 
This is given by  applying the stability property
of weighted dependency graphs of \cref{prop:WDG_Products} to the specific
dependency graph of \cref{prop:WDG_conjugacy},
as we now explain.

Let us recall the notation from \cref{ssec:product_WDG}, adapted to the present situation.
If each of $\alpha_1,\dots,\alpha_\ell$ is a list of $k$ elements (or $k$-list for short) of $[n]^2$, we define
\[\bm{\Psi}(\alpha_1,\dots,\alpha_\ell) = \Psi(\alpha_1 \cup \dots \cup \alpha_\ell),\]
where in the right-hand side, lists are seen as sets, and we recall that $\Psi$ is defined in \eqref{eq:def_Psi}.
   We also consider the $k$-th power $L^k$  of the weighted graph $L$
from \cref{prop:WDG_conjugacy}.
Concretely, $L^k$ has vertices indexed by $k$-lists of $[n]^2$ and its edge weights are given by
\[  w_k(\alpha,\beta) := \max_{(i,\ell) \in \alpha,\, (j,\ell') \in \beta} w((i,\ell),(j,\ell'))
=\begin{cases}
  \frac1n &\text{ if  $S(\alpha) \cap S(\beta) =\emptyset$;}\\
  1 &\text{ otherwise.}
\end{cases}
\]
Using \cref{prop:WDG_Products},
  there is a universal sequence $\bm D_k=(C^k_r)_{r \ge 1}$ such that
 $L^k$ is a $(\bm{\Psi}, \bm D_k )$ weighted dependency graph
 for the set of random variables $\prod_{t=1}^k B_{i_t,j_t}$, indexed
 by $k$-lists $\alpha=((i_1,j_1),\dots,(i_k,j_k))$ 
  of elements of $[n]^2$.
 We shall consider the restriction of this weighted graph
 to $k$-lists in $\cI^{(\pi,A)}$:
 this is a $(\bm{\Psi}, \bm D_k )$ weighted dependency graph
 for the set of random variables $\prod_{t=1}^k B_{i_t,j_t}$, indexed
 by $\cI^{(\pi,A)}$.
 
 To use \cref{lem:BorneCumulant}, we want to bound the quantities
 $R$ and $(T_\ell)_{\ell \geq 1}$,
 whose definitions are now recalled for this specific weighted dependency graph.
 \begin{equation}\label{eq:R_substituted}
R =\sum_{\alpha} \bm\Psi (\alpha) = \sum_{(i_1,\ldots,i_k,j_1,\ldots,j_k) \in \cI^{(\pi,A)}}  n^{\CC\big( (i_1,j_1),\dots,(i_k,j_k) \big) - \#\{ i_1,\ldots,i_k,j_1,\ldots,j_k\}} ,
 \end{equation}
 and 
   \begin{equation}\label{eq:def_Tl}
    T_\ell = \max_{\alpha_1, \ldots, \alpha_\ell} \left[\sum_{\beta} W(\{\beta\}, \{\alpha_1, \ldots, \alpha_\ell\}) 
    \frac{\bm{\Psi}(\alpha_1, \ldots, \alpha_\ell, \beta)}{\bm{\Psi}(\alpha_1, \ldots, \alpha_\ell)} \right],
  \end{equation}
    where, in the indices, each of $\alpha_1, \ldots, \alpha_\ell, \beta$
   is a $k$-list in $\cI^{(\pi,A)}$, and where
\begin{equation}\label{eq:W_simplified}
  W(\{\beta\}, \{\alpha_1, \ldots, \alpha_\ell\}) = \max_i w_k(\beta, \alpha_i) = \begin{cases}
  \frac1n &\text{ if  $S(\beta) \cap (S(\alpha_1) \cup \cdots \cup S(\alpha_\ell)) =\emptyset$;}\\
  1 &\text{ otherwise.}
\end{cases}
\end{equation}

We start by bounding $R$.
\begin{lemma}
    $R \le 2^{\binom{2k}{2}} n^m$.
\end{lemma}
\begin{proof}
 As in \cref{sec:WDG}, we introduce two graphs $G^{(I)}$ and $G^{(A)}$ associated
with a sequence $(i_1,\ldots,i_k,j_1,\ldots,j_k)$
(the dependence in $(i_1,\ldots,i_k,j_1,\ldots,j_k)$ is implicit in the notation).
 Both have vertex set $[k] \times \{1,2\}$,
where the vertex $(s,1)$, resp.~$(s,2)$, is labelled with $i_s$, resp.~$j_s$. 
Edges are chosen as follows:
\begin{itemize}
 \item in $G^{(I)}$, we connect $(s,1)$ and $(s,2)$ for all $s \le k$;
  \item in $G^{(A)}$, we connect $(s,1)$ and $(s+1,1)$ for {$s \in A$};
 \item in both graphs, we additionally connect vertices with equal labels.
\end{itemize}
With the same arguments as \cref{lem:ML_cc}, we see that connected components
of $G^{(I)}$ correspond to that of $(S((i_1,j_1),\dots,(i_k,j_k)),\{(i_1,j_1),\dots,(i_k,j_k)\})$,
implying 
\[\cc(G^{(I)})=\CC\big( (i_1,j_1),\dots,(i_k,j_k) \big).\]
Furthermore, let $G^{(\cap)}$ be the intersection of $G^{(I)}$ and $G^{(A)}$;
it only has edges connecting vertices with equal labels, and hence its number of components is the number of distinct values in $\{ i_1,\ldots,i_k,j_1,\ldots,j_k\}$.
Thus \eqref{eq:R_substituted} can be rewritten as
\begin{equation}\label{eq:first_bound_R}
 R =\sum_{(i_1,\ldots,i_k,j_1,\ldots,j_k) \in \cI^{(\pi,A)}}  
n^{\cc(G^{(I)})-\cc(G^{(\cap)})}.
\end{equation}
We now introduce the union $G^{(\cup)}$ of $G^{(I)}$ and $G^{(A)}$ 
(union in terms of edge-set). 
Since it contains the edges $\{((s,1),(s,2)),\, s \le k\} \cup \{( (s,1),(s+1,1)),\, s \in A\}$,
it has at most $m=k-\#A$ connected components.
Using \cref{eq:first_bound_R,lem:Inclusion_Exclusion_Connected_Component},
we get
\[R \le \sum_{(i_1,\ldots,i_k,j_1,\ldots,j_k) \in \cI^{(\pi,A)}}  
n^{\cc(G^{(\cup)})-\cc(G^{(A)})} \le\, n^m \sum_{(i_1,\ldots,i_k,j_1,\ldots,j_k) \in \cI^{(\pi,A)}} n^{-\cc(G^{(A)})}.\]
But, for a given graph $G$, there are at most $n^{\cc(G)}$ lists $(i_1,\ldots,i_k,j_1,\ldots,j_k)$
in $\cI^{(\pi,A)}$ such that $G^{(A)}=G$. 
Indeed, choosing a value $i_s$ or $j_s$ corresponding to a vertex of $G$
forces the values of variables $i_t$ or $j_t$ in the same component.
We finally get $R \le 2^{\binom{2k}{2}} n^m$.
\end{proof}
\medskip

The next step is to bound $T_\ell$ for all $\ell \ge 1$.
\begin{lemma} \label{lem:bound_T}
For any $\ell\geq 1 $, there exists $C^{''}_{k,\ell}$ (depending on the size $k$ of the pattern $(\pi,A)$ and on $\ell$, but not on $n$)
such that , the following holds true $$T_\ell \leq C^{''}_{k,\ell} n^{m-1}.$$
\end{lemma}

To prove this lemma, we need the following simple combinatorial fact to control
the various terms in \eqref{eq:def_Tl}.
\begin{lemma}\label{lem:diff_CC}
For any $k$-lists $\alpha_1, \ldots, \alpha_\ell$ and $\beta=((i_1,j_1),\dots,(i_k,j_k))$ of elements of $[n]^2$, 
one has
\[
\CC(\alpha_1 \cup \dots \cup  \alpha_\ell \cup\beta) - \CC(\alpha_1 \cup \dots \cup  \alpha_\ell) \leq k - I(\beta),
\]
where $I(\beta)$ is the number of pairs $(i_s,j_s)$ in $\beta$ which intersect $S(\alpha_1 \cup \dots \cup  \alpha_\ell)$.
\end{lemma}

\begin{proof}
We can see the graph with edge set
 $\alpha_1 \cup \dots \cup  \alpha_\ell \cup\beta$ (and no isolated vertex)
  as obtained from  that with edge set  $\alpha_1 \cup \dots \cup  \alpha_\ell$
  adding successively the edges $(i_1,j_1)$, \dots, $(i_k,j_k)$ (with their extremity as new vertices if needed). 
  When adding the edge $(i_s,j_s)$, if $(i_s,j_s)$ intersects $S(\alpha_1 \cup \dots \cup  \alpha_\ell)$,
  the number of components does not increase, otherwise it might increase by $1$.
  This proves the lemma.
\end{proof}

\begin{proof}[Proof of \cref{lem:bound_T}]
    We first fix $\alpha_1, \ldots, \alpha_\ell$ and write $S(\bm{\alpha})= S(\alpha_1) \cup \cdots \cup S(\alpha_\ell)$.
    We split the sum in \eqref{eq:def_Tl} according to
the intersection $S(\beta) \cap S(\bm{\alpha})$ being empty or not.
When $S(\beta) \cap S(\bm\alpha)=\emptyset$,
we have $W(\{\beta\}, \{\alpha_1, \ldots, \alpha_\ell\})=\frac{1}{n}$ (\Cref{eq:W_simplified})
and $I(\beta)=0$. Thus \cref{lem:diff_CC} gives
\[\frac{\bm{\Psi}(\alpha_1, \ldots, \alpha_\ell, \beta)}{\bm{\Psi}(\alpha_1, \ldots, \alpha_\ell)}
=n^{\CC(\alpha_1 \cup \dots \cup \alpha_\ell \cup \beta) - \CC(\alpha_1 \cup \dots \cup \alpha_\ell) - (\# S(\alpha_1 \cup \dots \cup \alpha_\ell,\beta) - \# S(\alpha_1 \cup \dots \cup \alpha_\ell))}  \le n^{k-\#S (\beta)}.\]
This yields the following bound:
\[
  \sum_{\beta : S(\beta) \cap S(\bm{\alpha}) =\emptyset}
W(\{\beta\}, \{\alpha_1, \ldots, \alpha_\ell\}) \frac{\bm{\Psi}(\alpha_1, \ldots, \alpha_\ell, \beta)}{\bm{\Psi}(\alpha_1, \ldots, \alpha_\ell)} 
\le \sum_{i=1}^{2k} \sum_{\beta :\# S(\beta)=i}
 n^{k-i-1}.\]
 But the number of $k$-lists $\beta$ in $\cI^{(\pi,A)}$ with $\# S(\beta)=i$ is bounded by
 $C_{k,i} n^{i-|A|}$ for some constant $C_{k,i}$ (to construct such a $\beta$,
 one chooses $i-|A|$ values, the other being forced by the adjacencies conditions $i_{s+1}=i_s+1$ for $s \in A$).
 This yields, recalling that $m=k-\#A$,
 \[\sum_{\beta : S(\beta) \cap S(\bm{\alpha}) =\emptyset}      
 W(\{\beta\}, \{\alpha_1, \ldots, \alpha_\ell\}) \frac{\bm{\Psi}(\alpha_1, \ldots, \alpha_\ell, \beta)}{\bm{\Psi}(\alpha_1, \ldots, \alpha_\ell)} \le \left(\sum_{i=1}^k C_{k,i}\right) n^{m-1}.\]
\smallskip

Let us now consider the summands in \eqref{eq:def_Tl} indexed by terms $\beta$ in $\cI^{(\pi,A)}$
 such that $S(\beta) \cap S(\bm\alpha) \ne \emptyset$.
In this case, we have $W(\{\beta\}, \{\alpha_1, \ldots, \alpha_\ell\})=1$  (see \Cref{eq:W_simplified})
and the analysis is more subtle.
\smallskip

We write as usual $\beta=( (i_1,j_1),\cdots,(i_k,j_k) )$.
First we recall that the constraints $i_{s+1}=i_s+1$ for $s \in A$ splits the sequence $(i_1,\cdots,i_k)$ into $m$
blocks of necessary consecutive values.
This allows to define $F_\alpha(\beta)$ ($F$ stands for {\em forced}) as the number of indices $s$ such that $i_s$ is not in $S(\alpha)$,
but either $i_{s-1}$ is in the same block as $i_s$ (i.e.~$i_s-1 \in A$),
 or one $i_t$ in the same block is in $S(\alpha)$.
Then, for fixed $\alpha,i,j$, the number of lists $\beta$ in $\cI^{(\pi,A)}$ with $F_{\alpha}(\beta)=j$
 and $\# \big(S(\beta) \setminus                
S(\alpha_1 \cup \dots \cup \alpha_{\ell}) \big)=i$ 
is bounded by                                   
$C'_{k,\ell,i,j} n^{i-j}$ for some constant $C'_{k,\ell,i,j}$.
Indeed, to construct $\beta$, one needs $i$ new values, 
but $j$ of them are forced by the constraints $i_{s+1}=i_s$ for $s \in A$.
\smallskip

On the other hand, in each block of indices, all but possibly one $i_s$ in the block are in $F_\alpha(\beta) \uplus S(\alpha)$.
Therefore, $|F_\alpha(\beta)|+|S(\alpha) \cap \{i_1,\dots,i_k\}|\ge k-m$.
Equality occurs if and only if $|S(\alpha) \cap \{i_1,\dots,i_k\}|=0$.
Looking at the definition of $I(\beta)$ in \cref{lem:diff_CC},
we trivially have~$|I(\beta)| \ge |S(\alpha) \cap \{i_1,\dots,i_k\}|$.
Putting both inequalities together, we have \hbox{$|F_\alpha(\beta)|+|I(\beta)| \ge k-m$}, with equality if and only if $|I(\beta)|=0$.
But $|I(\beta)|=0$ implies $S(\beta) \cap S(\bm\alpha)= \emptyset$.
Since we are considering the case $S(\beta) \cap S(\bm\alpha)\ne\emptyset$, the inequality above must be strict
and we have in fact $|F_\alpha(\beta)|+|I(\beta)| \ge k-m+1$.
\smallskip

Using the definition of $\Psi$ (\cref{eq:def_Psi}), \cref{lem:diff_CC} and the inequality above, we have
\[
  \frac{\bm{\Psi}(\alpha_1, \ldots, \alpha_\ell, \beta)}{\bm{\Psi}(\alpha_1, \ldots, \alpha_\ell)}
\le n^{k-|I(\beta)| - \#S(\beta \setminus (\alpha_1 \cup \dots \cup \alpha_\ell))}
\le n^{|F_\alpha(\beta)| +m-1- \#S(\beta \setminus (\alpha_1 \cup \dots \cup \alpha_\ell))}.
\]
This yields the following bound:
\[  \sum_{\beta : S(\beta) \cap S(\bm\alpha) \ne \emptyset}
W(\{\beta\}, \{\alpha_1, \ldots, \alpha_\ell\}) \frac{\bm{\Psi}(\alpha_1, \ldots, \alpha_\ell, \beta)}{\bm{\Psi}(\alpha_1, \ldots, \alpha_\ell)}\\
\le \sum_{ i\le 2m-1 \atop j \le |A|} \sum_{\beta : \# (S(\beta) \setminus S(\alpha_1 \cup \dots \cup \alpha_{\ell}) ) =i \atop 
|F_\alpha(\beta)| =j}
 n^{j+m-1-i}.
\]
Recalling that there are at most $C'_{k,\ell,i,j} n^{i-j}$ terms in the latter sum, we conclude that
\[\sum_{\beta : S(\beta) \cap S(\bm\alpha) \ne \emptyset}      
 W(\{\beta\}, \{\alpha_1, \ldots, \alpha_\ell\}) \frac{\bm{\Psi}(\alpha_1, \ldots, \alpha_\ell, \beta)}{\bm{\Psi}(\alpha_1, \ldots, \alpha_\ell)} \le \left(\sum_{ i\le 2m-1, j \le |A|}^m C'_{k,\ell,i,j}\right) n^{m-1}.\]
Bringing everything together, we get that
\[ \sum_{\beta}      
 W(\{\beta\}, \{\alpha_1, \ldots, \alpha_\ell\}) \frac{\bm{\Psi}(\alpha_1, \ldots, \alpha_\ell, \beta)}{\bm{\Psi}(\alpha_1, \ldots, \alpha_\ell)} \le C_{k,\ell}'' n^{m-1}.\]
 Since this holds for any tuple $(\alpha_1, \ldots, \alpha_\ell)$,
 we have $T_\ell \leq C_{k,\ell}'' n^{m-1}$, as wanted.
 \end{proof}

Back to the proof of \Cref{thm:clt}, \cref{lem:BorneCumulant} gives us
\[\left|
\kappa_r(X^{(\pi,A)}(\sigma_{\lambda^n}))\right|  \leq D_r R\prod_{i=1}^{r-1} T_i \leq C_{k,r}''' n^{r(m-1)+1},
\]
for some constants $C_{k,r}'''$.
Consequently, if $r>2$, we get:
\[
\kappa_r \left(\frac{X^{(\pi,A)}(\sigma_{\lambda^n}) - \E(X^{(\pi,A)}(\sigma_{\lambda^n}))}{n^{m-\frac 12}}\right) = O(n^{1-\frac{r}{2}}) \to 0.
\]
This concludes the proof of \Cref{thm:clt}.

\section{Non-degeneracy of the limiting law}
\label{sec:non-degeneracy}
In this section, we prove \cref{thm:non-degeneracy}. 
Recall that this theorem only considers classical patterns, i.e.~the case $A=\emptyset$,
and that we write $\pi$ instead of $(\pi,\emptyset)$ for simplicity.
Note that in this case, the parameter $m$ in \cref{thm:clt} is simply equal to the size $k$ of the pattern.

Let us define $p_3=\tfrac13 (1-p_1-2 p_2) \ge 0$.
For $n \ge 0$, we set $m_i=\lfloor np_i \rfloor$ for $i\in \{2,3\}$ and $m_1=n-2m_2-3m_3$,
so that $(1^{m_1},2^{m_2},3^{m_3})$ is a partition of $n$, which we denote by $\la^n$.
From \cite[Theorem 8.14]{hamaker}, it holds that
\begin{equation}  \label{eq:VarianceHamaker}
\Var\big( X^{\pi}(\sigma_{\lambda^n}) \big) = \big(f(p_1) +p_2 g(p_1) \big) \, n^{2k-1} + O(n^{2k-2}).    
\end{equation}
In fact, in \cite[Theorem 8.14]{hamaker}, the error term is $o(n^{2k-1})$
and this is best possible assuming only $m_i/n \to p_i$,
 but a quick inspection
of the proof reveals that if $m_i/n=p_i+O(1/n)$ for $i \in \{1,2\}$,
then the error term is $O(n^{2k-2})$, as claimed.

We shall prove that $\Var\big( X^{\pi}(\sigma_{\lambda^n}) \big) \ge K_1 n^{2k-3/2}$,
for some constant $K_1=K_1(p_1,p_2)>0$. Comparing with \cref{eq:VarianceHamaker},
this forces $f(p_1) +p_2 g(p_1) >0$.
The method of proof is inspired from
\cite[Section 4]{Lisa} and \cite[Section 3.4]{WDG_SetPartitions}.
The idea is to establish a recursive inequality on $\Var\big( X^{\pi}(\sigma_{\lambda^n}) \big)$,
using a recursive construction of $\sigma_{\lambda^n}$ and the law of total variance.
We first consider the case $p_2>0$ (in \cref{ssec:varphi,ssec:recursive_variance,ssec:initial_bound,ssec:recursive_ineq})
and indicate later (in \cref{ssec:no_cycle2}) the appropriate modifications in the case $p_2=0, p_1<1$.

\subsection{Preliminary: one point conditional pattern densities}
\label{ssec:varphi}
In this section, we introduce some quantity which will be useful later.
Let $P_i=(x_i,y_i)$, for $i \le k$, be points in the unit square with distinct coordinates,
i.e.~the $x_i$'s are distinct and the $y_i$'s are distinct, but we allow here some $x_i$ to be equal to some $y_j$.
We reorder the points $(P_i)_{i \le k}$ as $P_{(i)}=(x_{(i)},y_{(i)})$ such that
$x_{(1)}<\dots<x_{(k)}$. Then there exists a unique permutation $\pi$
such that $y_{(\pi^{-1}_1)} < \dots < y_{(\pi^{-1}_k)}$,
and we set $\Perm(P_1,\dots,P_k):=\pi$.
For example,  if $P_1=(0.1,0.5), \, P_2=(0.7,0.2)$ and $P_3=(0, 0.33)$ then 
 $P_{(1)}=P_3=(0,0.333),$ $\  P_{(2)}=P_1=(0.1,0.5),P_{(3)}=P_2=(0.7,0.2)$ and 
$\Perm(P_1,P_2,P_3)=231$.

Let us introduce some randomness. We denote by $\Lambda$ the Lebesgue measure on $[0,1]^2$
and by $\Delta$ that on the diagonal $\{(x,x), x \in [0,1]\}$ with total mass 1. In particular, for any $p_1\in [0,1]$, $p_1 \Delta + (1-p_1) \Lambda$ has uniform marginals. 
We then let $P_1,\dots,P_{k-1}$ be i.i.d.~random points
with distribution $p_1 \Delta + (1-p_1) \Lambda$ and set 
\begin{equation}\label{eq:def_varphi}
  \varphi^\pi_{p_1}(x,y) = \mathbb P\big[ \Perm(P_1,\dots,P_{k-1},(x,y))=\pi \big].
\end{equation}

\begin{lemma}
  Assume $\pi$ has size at least $2$ and $p_1 <1$.
  Then the function
  \[(x,y) \mapsto \varphi^\pi_{p_1}(x,y) + \varphi^\pi_{p_1}(y,x)\]
  is non constant.
  \label{lem:varphi_non_constant}
\end{lemma}
\begin{proof}
Let $\widetilde{\pi}$ be the pattern obtained from $\pi$ by removing the first entry, and standardizing.
We look at $\varphi^{\pi}_{p_1}((0,\tfrac{1}{2}))$, i.e~the probability that random
points $P_1,\dots,P_{k-1}$ together with $(0,\tfrac{1}{2})$ induce the pattern $\pi$.
This event is equivalent to the fact that $P_1,\dots,P_{k-1}$ induce $\widetilde{\pi}$
and that exactly $\pi_1-1$ points among $P_1,\dots,P_{k-1}$ have a $y$-coordinate
smaller than $\tfrac{1}{2}$. Hence, we have
\begin{align*}
{{\varphi}}^{\pi}_{p_1}(0,\tfrac{1}{2})&=
\mathbb P\big[ \Perm(P_1,\dots,P_{k-1})=\widetilde\pi \ \wedge \ \#\{i \le k-1:\, y_i<\tfrac12 \}=\pi_1-1 \big]
\\ &\geq  \mathbb P\big[ \Perm(P_1,\dots,P_{k-1})=\widetilde\pi,  \ \wedge \ \#\{i \le k-1:\, y_i<\tfrac12 \}=\pi_1-1 \ \wedge\ \forall i, x_i \neq y_i  \big].
\end{align*}
The latter event \enquote{$ \forall i, x_i \neq y_i$}
is a.s.~equivalent to the $P_i$'s being sampled according to $\Lambda$
and hence has probability $(1-p_1)^{k-1}$ (recall that  $P_1,\dots,P_{k-1}$
 have distribution $p_1 \Delta + (1-p_1) \Lambda$).
 Conditionally to \enquote{$ \forall i, x_i \neq y_i$}, the events
 \[\Perm(P_1,\dots,P_{k-1})=\widetilde\pi\,  \text{ and }
  \ \#\big\{i \le k-1:\, y_i<\tfrac12 \big\} =\pi_1-1\] are independent
  (in general, the pattern induced by uniform random points
  is independent of the set of its $y$-coordinates).
  These events occur with probability $\frac{1}{(k-1)!}$
  and $\binom{k-1} {\pi_1-1} \,2^{-k+1}$ respectively.
  Thus we get
  \[{{\varphi}}^{\pi}_{p_1}(0,\tfrac{1}{2})
\ge  \frac{(1-p_1)^{k-1}}{(\pi_1-1)! \, (k-\pi_1)!\, 2^{k-1}}. \]
In particular, when $p_1<1$, we have ${{\varphi}}^{\pi}_{p_1}(0,\frac{1}{2}) >0$.
On the other hand,
\begin{itemize}
    \item If $\pi_1\neq 1$, then ${{\varphi}}^{\pi}_{p_1}(0,0)=0$; 
    \item If $\pi_1= 1$ and if $\pi$ has size at least $2$,
     then ${{\varphi}}^{\pi}_{p_1}(0,1)={{\varphi}}^{\pi}_{p_1}(1,0)=0$.
\end{itemize}
In both cases, we can conclude that the function 
$\varphi^\pi_{p_1}(x,y) + \varphi^\pi_{p_1}(y,x)$ is not constant.
\end{proof}

We also note that
$|\varphi^\pi_{p_1}(x,y) - \varphi^\pi_{p_1}(x',y')|$ is bounded from above
by the probability that at least one of $P_1,\dots,P_{k-1}$ has its $x$-coordinate between $x$ and $x'$ or 
its $y$-coordinate between $y$ and $y'$. Using a union bound,   this implies
\[|\varphi^\pi_{p_1}(x,y) - \varphi^\pi_{p_1}(x,y)| \le (k-1) (|x-x'|+|y-y'|),\]
i.e.~$\varphi^\pi_{p_1}$ is $(k-1)$-Lipschitz when using the $L_1$ norm on $[0,1]$.

\subsection{A recursive inequality on the variance}
\label{ssec:recursive_variance}
Let $\la$ be a partition with at least one part equal to $2$ and denote by $\la'$ the partition obtained by
removing a part of length 2 to $\la$.
Then a uniform random permutation $\si=\si_\la$ in $\mathcal C_{\la}$ can be constructed as follows.
\begin{itemize}
  \item Choose two distinct indices $i$ and $j$ uniformly at random between $1$ and $n$,
    and set $\si_i=j$ and $\si_j=i$
  \item Take a uniform random permutation $\si'$ of $[n] \setminus \{i,j\}$ of cycle type $\la'$,
    independent of $\{i,j\}$, and set $\si_k=\si'_k$ for $k \notin \{i,j\}$.
\end{itemize}
Using the law of total variance, we have
\[v(\la):=\Var(X^{\pi}(\si)) = \esper\big[ \Var(X^{\pi}(\si)|i,j) \big]
    + \Var\big[ \esper( X^{\pi}(\si)|i,j) \big].\]
Moreover, the quantity $X^{\pi}(\si)$ can be decomposed 
as $X^{\pi}(\si) = X^{\pi}(\si') +C$, where
  $C$ is the number of occurrences of $\pi$ in $\si$,
    using position $i$ or $j$ (or both).
Note that $X^{\pi}(\si')$ is independent of $i,j$, 
so that $\esper(X^{\pi}(\si')|i,j)$ is the {\em constant} random variable
a.s.~equal to $\esper(X^{\pi}(\si'))$. This implies that
\[ \Var\big[ \esper(X^{\pi}(\si)|i,j) \big] 
= \Var\big[ \esper(X^{\pi}(\si')) + \esper(C|i,j) \big] = \Var\big[\esper(C|i,j) \big].\]
In particular, we get an initial bound
\begin{equation}
  v(\la) \ge \Var\big[\esper(C|i,j) \big].
  \label{eq:InitialBound}
\end{equation}
On the other hand, from Cauchy-Schwarz and Jensen inequalities,
since $X^{\pi}(\si')$ and $i,j$ are independent,
we have
\[ \Big| \esper\big[\Cov(X^{\pi}(\si'),C |i,j)\big] \Big|
\le \esper\big[ \sqrt{\Var(X^{\pi}(\si'))} \sqrt{\Var(C|i,j)} \big] 
\le \sqrt{\Var(X^{\pi}(\si'))} \, \sqrt{\esper[\Var(C|i,j) ]}.\]
Expanding $\Var(X^{\pi}(\si)|i,j)=\Var(X^{\pi}(\si')+C|i,j)$ by bilinearity, we find that
\begin{align*}
  \esper\big[ \Var(X^{\pi}(\si)|i,j) \big] &
= \Var(X^{\pi}(\si')) + 2\esper\big[\Cov(X^{\pi}(\si'),C |i,j)\big] + \esper\big[\Var(C|i,j)\big] \\
&\ge \Var(X^{\pi}(\si')) -2 \sqrt{\Var(X^{\pi}(\si'))} \sqrt{\esper[\Var(C|i,j)]}.
\end{align*}
Note that $\Var(X^{\pi}(\si'))=v(\la')$,
since $\si'$ is a uniform random permutation of cycle-type $\la'$.
Summing up, we get the following recursive inequality on $v(\la)$:
\begin{equation}
  v(\la) \ge v(\la') \, \bigg( 1- 2\sqrt{\tfrac{ \esper[\Var(C|i,j)]}{v(\la')} }\, \bigg)
  + \Var\big[\esper(C|i,j) \big]. 
  \label{eq:RecursiveInequality}
\end{equation}

\subsection{Analysing the initial bound \eqref{eq:InitialBound}}
\label{ssec:initial_bound}
Recall that $C$ counts the number of occurrences of $\pi$ in $\si$ that use position $i$ or $j$ (or both).
We first consider the number $C_1$ of those using $i$ but not $j$.
\begin{lemma}
  Let $x,y$ be in $[0,1]$ and $i=i(n)$ and $j=j(n)$ be two sequences chosen such that $i=xn+o(n)$ and $j=yn+o(n)$.
  Then we have
  \[ \esper(C_1|i,j) = \frac{1}{(k-1)!}n^{k-1} \varphi^\pi_{p_1}(x,y) +o(n^{k-1}),\]
  where $\varphi^\pi_{p_1}$ is defined in \cref{eq:def_varphi} and the error term is uniform on $i$ and $j$.
  \label{lem:EC1}
\end{lemma}
We start with some estimates on the distribution of
finitely many values $(\sigma(i_1),\dots,\sigma(i_m))$
of a uniform random permutation $\sigma$ of cycle type $\la$.
\begin{lemma}
\label{lem:estimates-sigma-i1-im}
  \begin{enumerate}
    \item Fix $m \ge 1$ and let $i_1,\dots,i_m$ be distinct indices in $\{1,\dots,n\}$.
      Then 
      \[\mathbb{P}(\exists s \ne t : \sigma_{i_s} = i_t) =O(n^{-1}).\]
    \item Let $j_1,\dots,j_m$ also be distinct indices in $\{1,\dots,n\}$.
      We assume that there is no $s \ne t$ such that $i_s=j_t$ and we let $r=|\{t: i_t=j_t \}|$. Then
  \[ \mathbb{P}( \forall s \le m,\, \sigma_{i_s} = j_s ) = \left( p_1 \right)^r  \big(\tfrac{1-p_1}{n}\big)^{m-r}(1 +O(n^{-1})).\]  
  \end{enumerate}
  In both estimates, the error terms depend on $m$, 
  but are uniform in $(i_1,\dots,i_m,j_1,\dots,j_m)$ and in $\la$.
  Moreover, the same estimates hold conditionally to $i,j$ in the construction of \cref{ssec:recursive_variance}, provided that
 the tuple $(i_1,\dots,i_m,j_1,\dots,j_m)$ does not contain $i$ or $j$.
\end{lemma}
\begin{proof}
  For the first item, let us note that a given index $i_s$ is a fixed point of $\sigma$ with probability $m_1(\la)/n$.
  Conditionally of $i_s$ not being a fixed point, by conjugation invariance, its image $\sigma_{i_s}$ is uniformly distributed
  on $[n] \setminus \{i_s\}$, and thus for fixed $s \ne t$, we have
  \[ \mathbb{P}( \sigma_{i_s} = i_t ) = \frac{n-m_1(\la)}{n} \cdot \frac1{n-1} \le \frac{1}{n-1}.\]
  The first item in the lemma follows by a simple union bound.

Let us consider the second item.
By symmetry, we can assume that $i_s=j_s$ for $s \le r$ and $i_s \ne j_s$ for $s >r$.
Recalling that the set $F(\sigma)$  of fixed points of $\sigma$ is
 a uniform random subset of $[n]$ of size $m_1(\la)$, we have
\[ \mathbb P \big[  \forall s \le r,\, \sigma(i_s) =i_s \big] = \mathbb P\big[ \{i_1,\dots,i_r\} \subset F(\sigma) \big]
= \frac{\binom{n-r}{m_1-r}}{\binom{n}{m_1}} =  p_1^r +O(n^{-1}). \]
Now, conditionally on $i_1,\dots,i_r$ being fixed points, the probability that
$i_{r+1},\dots,i_m$ are {\bf not} fixed points is
\[ \frac{\binom{n-m}{m_1-r}}{\binom{n-r}{m_1-r}}  = (1-p_1)^{m-r} + O(n^{-1}).  \]
Combining this computation with item i), the probability of $i_1,\dots,i_r$ being fixed points, 
of $i_{r+1},\dots,i_m$ being not fixed points
and of $(\sigma_{i_{r+1}},\dots,\sigma_{i_m})$ being disjoint from $(i_{r+1},\dots,i_m)$ is
\[ p_1^r  (1-p_1)^{m-r} + O(n^{-1}).\]
But under this event,  by conjugation invariance, the tuple $(\sigma_{i_{r+1}},\dots,\sigma_{i_m})$ is a uniform random tuple of distinct values
in $[n] \setminus \{ i_1,\dots,i_m\}$. Hence, the conditional probability that is equal to $(j_{r+1},\cdots,j_m)$
is $1/n^{m-r} (1+O(n^{-1}))$.
We finally get
\[
\mathbb{P}( \forall s \le m,\, \sigma(i_s) = j_s ) = \frac{p_1^r  (1-p_1)^{m-r} + O(n^{-1})}{n^{m-r}(1 +O(n^{-1}))},\]
which proves the second item.

The conditional statement in the lemma follows easily from the unconditional
one since for $h \notin \{i,j\}$, we have $\sigma_h=\sigma'_h$,
where $\sigma'$ is a uniform random permutation of cycle-type $\lambda'$
of the set $[n] \setminus \{i,j\}$.
\end{proof}

\begin{proof}[Proof of \cref{lem:EC1}]
Recall that $C_1$ is the number of occurrences of $\pi$ in $\sigma$ using position
$i$ but not $j$. 
Denoting by $i_1,\dots,i_{k-1}$ the other positions of such an occurrence of $\pi$
(in any order, hence the symmetry factor in the next formula)
and $j_1,\dots,j_{k-1}$ their images by $\sigma$,
we can write $C_1$ as follows:
\begin{multline}C_1 =\frac{1}{(k-1)!}\sum_{i_1,\dots,i_{k-1} \in [n]\setminus\{i,j\}\atop{\text{distinct}}}
\sum_{j_1,\ldots,j_{k-1}\in [n]\setminus \{i,j\} \atop {\text{distinct}}}
\Bigg( \mathbbm{1}\Big[\Perm\big((\tfrac{i_1}{n},\tfrac{j_1}{n}),\dots, (\tfrac{i_{k-1}}{n},\tfrac{j_{k-1}}{n}), (\tfrac{i}{n},\tfrac{j}{n})\big)=\pi \Big] \\
   \cdot \mathbbm{1}\big[\sigma(i_\ell)=j_\ell, \forall 1\le \ell\le k-1) \big] \Bigg)\,. 
   \label{eq:decompo-C1}
   \end{multline}
   By \cref{lem:estimates-sigma-i1-im}, item i),
   for each $i_1,\dots,i_{k-1}$, the total contribution of terms with some $j_t \ne i_t$,
   but $j_t \in \{i_1,\dots,i_{k-1},i\}$, is $O(n^{-1})$ in expectation.
   Hence, the total contribution of such terms to the double sum in \eqref{eq:decompo-C1}
   is $O(n^{k-2})$ in expectation.
   
   Furthermore, from \cref{lem:estimates-sigma-i1-im}, item ii), if $(j_1,\dots,j_{k-1})$ is such that for each $t$, either $j_t=i_t$
   or $j_t \notin \{i_1,\dots,i_{k-1},i\}$, we have, conditionally to $i,j$,
  $$ \mathbb P\big[\sigma(i_\ell)=j_\ell, \forall 1\le \ell\le k-1 \mid i,j \big]= (p_1)^{\#\{\ell:i_\ell=j_\ell\}}  \big(\tfrac{1-p_1}{n}\big)^{\#\{\ell:i_\ell\neq j_\ell\}} (1+O(\frac{1}{n})).$$
This gives
\begin{multline}\mathbb{E}(C_1|i,j)=\sum_{i_1,\dots,i_{k-1} \in [n]\setminus\{i,j\}\atop{\text{distinct}}}
\sum_{j_1,\ldots,j_{k-1}\in [n]\setminus \{i,j\} \atop {\text{distinct}}}\Bigg( \mathbbm{1}\Big[\Perm\big((\tfrac{i_1}{n},\tfrac{j_1}{n}),\dots, (\tfrac{i_{k-1}}{n},\tfrac{j_{k-1}}{n}), (\tfrac{i}{n},\tfrac{j}{n})\big)=\pi \Big] \\
   \cdot (p_1)^{\#\{\ell:i_\ell=j_\ell\}} \big(\tfrac{1-p_1}{n}\big)^{\#\{\ell:i_\ell\neq j_\ell\}}  \Bigg) +O(n^{k-2}).  
   \label{eq:C1_ij}
   \end{multline}
   We let  $(P^{(n)}_i)_{i \le k-1}$ be 
 a vector of independent random variables with the following distribution, conditioned not to share coordinates:
\begin{itemize}
\item with probability $p_1$, we set $P^{(n)}_1=(h/n,h/n)$, where $h$ is uniform in $[n] \setminus \{i,j\}$;
\item with probability $1-p_1$, we set $P^{(n)}_1=(h/n,h'/n)$, where $h,h'$ are uniform in $[n] \setminus \{i,j\}$ conditioned to satisfy $h \ne h'$;
\end{itemize}
Using this, \eqref{eq:C1_ij} rewrites as
\[\mathbb{E}(C_1|i,j)=\frac{1}{(k-1)!} n^{k-1} \mathbb P\big[ \Perm\big(P^{(n)}_1,\dots, P^{(n)}_{k-1},(\tfrac{i}{n},\tfrac{j}{n})\big)=\pi \big] +O(n^{k-2}).\]
But the tuple $\big(P^{(n)}_1,\dots, P^{(n)}_{k-1},(\tfrac{i}{n},\tfrac{j}{n})\big)$
converges in distribution to $\big(P_1,\dots, P_{k-1},(x,y)\big)$ as in \cref{eq:def_varphi}.
Since $\big(P_1,\dots, P_{k-1},(x,y)\big)$ have distinct coordinates with probability 1, and
since the map $\One[\Perm\big(\cdot,\cdots,\cdot\big)=\pi]$ is continuous on the set of $k$-tuples of points
with distinct coordinates, we get
\[\mathbb{E}(C_1|i,j)=\frac{1}{(k-1)!} n^{k-1} \mathbb P\big[ \Perm\big(P_1,\dots, P_{k-1},(x,y)\big)=\pi \big] +o(n^{k-1}).\]
It is easy to check that the error term is uniform in $i,j$, concluding the proof of the lemma.
\end{proof}

Symmetrically, letting $C_2$ be the number of occurrences of $\pi$ in $\si$ that use
the position $j$ but not $i$, we have
\[ \esper(C_2|i,j) = \frac{n^{k-1}}{(k-1)!} \varphi_{p_1}^\pi(y,x) +o(n^{k-1}).\]
Moreover, it is easy to see that there are less than $n^{k-2}$ occurrences of $\pi$
in $\si$ using simultaneously $i$ and $j$, so that
$C=C_1+C_2 + O(n^{k-2})$. Summing up, we get that, provided that $i=xn+O(1)$ and $j=yn+O(1)$,
\begin{equation}\label{eq:C_ij}
  \esper(C|i,j) = \frac{n^{k-1}}{(k-1)!} \big(\varphi_{p_1}^\pi(x,y)+\varphi_{p_1}^\pi(y,x) \big) + o(n^{k-1}).
\end{equation}
Since $i$,$j$ are a uniform random pair of distinct integers in $[n]$,
one can couple them with i.i.d.~random uniform variables $(U,V)$ 
in $[0,1]$ such that $i=Un+O(1)$ and $j=Vn+O(1)$.
Hence,
\[  \Var\big[\esper(C|i,j) \big] = \frac{n^{2k-2}}{( (k-1)!)^2} \Var\big( \varphi_{p_1}^\pi(U,V)+ \varphi_{p_1}^\pi(V,U) \big)+o(n^{2k-2}).\]
Using \cref{lem:varphi_non_constant} and the fact that $\varphi_{p_1}^\pi$ is $(k-1)$-Lipschitz,
we know that 
\[K_2:= \Var( \varphi_{p_1}^\pi(V,U)+ \varphi_{p_1}^\pi(U,V))>0.\]
Setting $K'_2=K_2/(k!)^2$,
\cref{eq:InitialBound} implies that for any $n$ sufficiently large,
\begin{equation}
  v(\la) \ge \Var\big[\esper(C|i,j) \big] \ge  {K'_2} {n^{2k-2}}.
  \label{eq:InitialBoundExplicit}
\end{equation}

\subsection{Improving the initial lower bound}~
Suppose that $p_2\neq0$.
\label{ssec:recursive_ineq}
The lower bound \eqref{eq:InitialBoundExplicit},
is not sufficient to  prove that the limiting variance does not vanish.
We will use the recursive inequality \cref{eq:RecursiveInequality} to improve it.
To this end, we first need to analyse the term $\esper\big[\Var(C|i,j)\big]$.

\begin{lemma}
There exist $\Kb,\Kbb>0$ such that for any $n>\Kb$,
\[\esper\big[\Var(C|i,j)\big] 
\le \Kbb n^{2k-3} .\]
  \label{lem:Tech}
\end{lemma}
\begin{proof}
Conditionally on $i,j$, we have
\begin{multline*}
  C 
  = \sum_{r=1}^k \, \sum_{i_1 < \dots  <i_k \in [n]\setminus\{j\} \atop i_r=i} \,
    \sum_{j_{\pi^{-1}(1)} < \dots < j_{\pi^{-1}(k)} \in [n]\setminus\{i\}  \atop j_r =j}
    {B_{i_r,j_r}} \prod_{\ell\neq r}
  B_{i_\ell,j_\ell} 
  \\+  
  \sum_{r=1}^k \, \sum_{i_1 < \dots  <i_k \in [n]\setminus\{i\} \atop i_r=j} \,
    \sum_{j_{\pi^{-1}(1)} < \dots < j_{\pi^{-1}(k)} \in [n]\setminus\{j\}  \atop j_r =i}
    {B_{i_r,j_r}} \prod_{\ell\neq r}
  B_{i_\ell,j_\ell}
\\+
\sum_{r_1=1}^k \sum_{r_2=1\atop r_2\neq r_1}^k \, \sum_{i_1 < \dots  <i_k \in [n] \atop i_{r_1}=i,i_{r_2}=j} \,
    \sum_{j_{\pi^{-1}(1)} < \dots < j_{\pi^{-1}(k)} \in [n]  \atop j_{r_1}=j,j_{r_2}=i
    }
 {B_{i_{r_1},j_{r_1}}}  {B_{i_{r_2},j_{r_2}}} \prod_{\ell\notin \{r_1,r_2\}}
  B_{i_\ell,j_\ell}.
\end{multline*}
Conditionally on $i,j$, we have, in the two first term $B_{i_r,j_r}=B_{i,j}=1$,
and in the third one $B_{i_{r_1},j_{r_1}}=B_{i,j}=1=B_{j,i}=B_{i_{r_2},j_{r_2}}$
The other $B$ factors have indices different from $i$ and $j$.
Since $\sigma'$
is a uniform permutation on the conjugacy class indexed by $\lambda'$
we have 
$(\bm{\Psi}, \bm D_{k-1} )$ weighted dependency graph
 for the random variables $\prod_{t=1}^{k-1} B_{i_t,j_t}$.

One can easily adapt the proof of \cref{sec:proof clt} to obtain that the quantities $R$ and $T_\ell$
associated with the above dependency graph satisfy $R\leq 3 2^{2(k-1)\choose 2} (n-2)^{k-1} $ and $T_\ell\leq K_{7,k} n^{k-2}$
(note that the product $\prod_{t=1}^{k-1} B_{i_t,j_t}$ under consideration have degree $k-1$).
From \cref{lem:BorneCumulant}, this implies $\Var(C|i,j) \leq K_4 n^{2k-3}$.
One can conclude since $K_4$ is independent of the choice of $i$ and $j$. 
\end{proof}
Plugging in \cref{eq:InitialBoundExplicit,lem:Tech} 
in \cref{eq:RecursiveInequality}, 
we get that, for some constant $\Kc>0$,
\begin{equation}
   v(\lambda) \ge v(\lambda') ( 1 - \Kc n^{-1/2})+\Kab n^{2k-2}. 
   \label{eq:RecursiveInequalityExplicit}
 \end{equation}
For, $i\le m_2(\lambda)$, we denote by $\lambda^{(i)}=((\lambda')\cdots)'$ the partition obtained by removing
$i$ blocks of size $2$ from $\lambda$.
For $i \le \Kd n^{1/2}$ (where $\Kd$ is a positive constant),
the inequality \eqref{eq:RecursiveInequalityExplicit} holds substituting $\la$ by $\la^{(i)}$ and we have
\[ v(\lambda^{(i)}) \ge v(\lambda^{(i+1)}) ( 1 - \Kc (n-2i)^{-1/2})+\Kab (n-2i)^{2k-2}.\]
We start from the initial inequality \eqref{eq:InitialBoundExplicit},
namely $v(\la^{(\Kd \sqrt{n})}) \ge \Kab (n-\Kd \sqrt{n})^{2k-2}$,
and iterate the above recursive inequality: we get
\[ v(\lambda) \ge \sum_{i=0}^{\Kd \sqrt{n}-1} \Kab (n-2i)^{2k-2} (1 - \Kc (n-2i)^{-1/2}) \cdots (1 - \Kc n^{-1/2})\]

For $\Kd$ sufficiently small, e.g.~$K_6=(2K_5)^{-1}$, all products $$(1 - \Kc (n-2i)^{-1/2}) \cdots (1 - \Kc n^{-1/2})$$
are bounded away from $0$, so that each of the $\Theta(\sqrt{n})$ terms in the sum
behaves as $\Theta(n^{2k-2})$.
Therefore, there exists a constant $\Kv>0$ such that
\begin{equation}
   v(\lambda) \ge \Kv n^{2k-3/2}.
   \label{eq:FinalInequalityVaraince}
 \end{equation}
Comparing with \eqref{eq:VarianceHamaker}, we conclude that $f(p_1)+p_2g(p_1)>0$.

\subsection{The case without cycles of size 2}~
 \label{ssec:no_cycle2}
When $p_2=0$ and $p_1<1$, the proof can be easily adapted. 
In such cases, one has $p_3=\tfrac13(1-p_1-2p_2)>0$. 
In particular, for $n>\frac{3}{1-p_1}$,   $\lambda$ contains a block of size $3$,
and we let $\hat{\lambda}$ be the partition obtained by removing a block of size $3$ from $\lambda$. One can construct $\sigma_\lambda$ as follows :
\begin{itemize}
    \item Choose  $i,j,h$ distinct  uniformly at random in $[n]$ and set $\sigma_i=j$, $\sigma_j=h$  and $\sigma_h=i$.
    \item Take a uniform random permutation $\hat{\sigma}$ of $[n]\setminus\{i,j,h\}$ of cycle type $\hat{\lambda}$ and set $\sigma_\ell=\hat{\sigma}_\ell$ for any $\ell\in [n]\setminus\{i,j,h\}$
\end{itemize}
Let $\widehat{C}$ be the number of occurrences of $\pi$ that contains $i$ or $j$ or $h$. 
The counterpart of \eqref{eq:C_ij} is the following.
\begin{lemma}
  Let $x,y,z$ be in $[0,1]$ and $i=i(n)$, $j=j(n)$ and $h=h(n)$ be three sequences of positive integers chosen such that $i=xn+o(n)$, $j=yn+o(n)$ and $h=zn+o(n)$.
  Then
  \[ \esper(\widehat{C}|i,j,h) = \frac{n^{k-1}}{(k-1)!} \big(\varphi_{p_1}^\pi(x,y) +\varphi_{p_1}^\pi(y,z) + \varphi_{p_1}^\pi(z,x)\big) +o(n^{k-1}),\]
  where $\varphi_{p_1}^\pi$ is defined in \cref{eq:def_varphi}.
\end{lemma}

The remainder of the proof is the same.
One need only to check the following variant of \cref{lem:varphi_non_constant}.
\begin{lemma}
  Assume $\pi$ has size at least $2$ and $p_1 <1$. Then the function
 $$f: (x,y,z) \mapsto \varphi_{p_1}^\pi (x,y) + \varphi_{p_1}^\pi (y,z) + \varphi_{p_1}^\pi (z,x)$$ is not constant. 
\end{lemma}
\begin{proof}
 Assume, for the sake of contradiction, that $f$ is constant.
Then $f(x, y, y) = f(x, x, y)$ for any $x$ and $y$.
This implies that $\varphi_{p_1}^\pi(x, x) = \varphi_{p_1}^\pi(y, y)$ for any $x$ and $y$,
i.e.~that the function $x \mapsto \varphi_{p_1}^\pi(x, x)$ is constant.
As a consequence, the expression $$(x, y) \mapsto f(x, x, y) - \varphi_{p_1}^\pi(x, x) = \varphi_{p_1}^\pi(x, y) + \varphi_{p_1}^\pi(y, x)$$ should also be constant,
leading to a contradiction with \cref{lem:varphi_non_constant}.
\end{proof}

\textbf{Acknowledgement:}
Both authors want to thank Zachary Hamaker and Victor Dubach for insightful discussions and
an anonymous referee for their careful reading of the article and their 
insightful comments.

Mohamed Slim Kammoun is supported by ERC Project LDRAM: ERC-2019-ADG Project 884584. 
Valentin Féray is partially supported by a {\em Future Leader} grant
from the initiative {\em Lorraine Université d'Excellence} (LUE).
\def\cprime{$'$}

\end{document}